\documentclass[11pt,a4paper]{article}
\usepackage{amsmath,amssymb,amsthm,amsfonts,latexsym,graphicx,subfigure}
\usepackage{fancyhdr,enumerate}

\usepackage[numbers,sort&compress]{natbib}

\usepackage{xcolor}

\DeclareMathOperator*{\sgn}{\ensuremath{sign}}
\DeclareMathOperator*{\vol}{\ensuremath{Vol}}
\newcommand{\rf}[1]{(\ref{#1})}

\newcommand{\RQ}{\mathrm{RQ}}
\newcommand{\Om}{\Omega}
\newcommand{\om}{\omega}
\newcommand{\baro}{\bar \omega}

\newcommand{\nequiv}{\equiv \hspace{-.4cm} \slash \hspace{.15cm}}
\newcommand{\diver}{\mathrm{div} \hspace{.03cm}}

\newcommand{\R}{\ensuremath{\mathbb{R}}}

\theoremstyle{plain}
\newtheorem{theorem}{Theorem}[section]
\newtheorem{lemma}[theorem]{Lemma}
\newtheorem{proposition}[theorem]{Proposition}
\newtheorem{corollary}[theorem]{Corollary}

\theoremstyle{definition}

\newtheorem{definition}[theorem]{Definition}

\theoremstyle{remark}
\newtheorem{remark}[theorem]{Remark}
\newtheorem{conjecture}{Conjecture}

\usepackage{authblk}

\begin{document}
\bibliographystyle{plain} 
\title{ $p$-Laplace Operators for Oriented Hypergraphs
}
\author[1,2]{J\"urgen Jost}
\author[1]{Raffaella Mulas}
\author[1]{Dong Zhang}
\affil[1]{Max Planck Institute for Mathematics in the Sciences, Leipzig, Germany}
\affil[2]{Santa Fe Institute for the Sciences of Complexity, Santa Fe, New Mexico, USA}

\date{}
\maketitle

\begin{abstract}The $p$-Laplacian for graphs, as well as the vertex Laplace operator and the hyperedge Laplace operator for the general setting of oriented hypergraphs, are generalized. In particular, both a vertex $p$-Laplacian and a hyperedge $p$-Laplacian are defined for oriented hypergraphs, for all $p\geq 1$. Several spectral properties of these operators are investigated.
\vspace{0.2cm}

\end{abstract}

\allowdisplaybreaks[4]

\section{Introduction}
Oriented hypergraphs are hypergraphs with the additional structure that each vertex in a hyperedge is either an input, an output or both. They have been introduced in \cite{Hypergraphs}, together with two normalized Laplace operators whose spectral properties and possible applications have been investigated also in further works \cite{Master-Stability,Sharp,MulasZhang,AndreottiMulas}. Here we generalize the Laplace operators on oriented hypergraphs by introducing, for each $p\in\R_{\geq 1}$, two $p$-Laplacians. While the vertex $p$-Laplacian is a known operator for graphs (see for instance \cite{Graph-p-Laplacian1,Graph-p-Laplacian2,Graph-p-Laplacian3}), to the best of our knowledge the only edge $p$-Laplacian for graphs that has been defined is the classical one for $p=2$.\newline

\textbf{Structure of the paper.} In Section \ref{section:Euclidean}, for completeness of the theory, we discuss the $p$-Laplacian on Euclidean domains and Riemannian manifolds, and in Section \ref{section:hyp} we recall the basic notions on oriented hypergraphs. In Section \ref{section:pLaplacians} we define the $p$-Laplacians for $p>1$ and we establish their generalized min-max principle, and similarly, in Section \ref{section:1Laplacians}, we introduce and discuss the $1$-Laplacians for oriented hypergraphs. Furthermore, in Section \ref{section:smallestlargest} we discuss the smallest and largest eigenvalues of the $p$-Laplacians for all $p$, in Section \ref{section:nodal} we prove two nodal domain theorems, and in Section \ref{section:lambdamin} we discuss the smallest nonzero eigenvalue. Finally, in Section \ref{section:vertexpartition} we discuss several vertex partition problems and their relations to the $p$-Laplacian eigenvalues, while in Section \ref{section:hyperedge} we discuss hyperedge partition problems.\newline

In \cite{pLaplacians2} we shall build upon the results developed  in this paper.\newline

\textbf{Related work.} It is worth mentioning that, in \cite{hein2013total},  
other vertex $p$-Laplacians for hypergraphs have been introduced and studied. While these generalized vertex $p$-Laplacians coincide with the ones that we introduce here in the case of graphs, they do not coincide for general hypergraphs. Also, \cite{hein2013total} focuses on classical hypergraphs, while we consider, more generally, oriented hypergraphs.


\subsection{The $p$-Laplacian on Euclidean domains and Riemannian manifolds}\label{section:Euclidean}
There is a strong analogy between Laplace operators on Euclidean domains and Riemannian manifolds on one hand and their discrete versions on graphs and hypergraphs, and this is also some motivation for our work. Therefore, it may be useful to briefly summarize the theory on Euclidean domains and Riemannian manifolds.

 Let $\Omega \subset \R^n$ be a bounded domain, with piecewise Lipschitz boundary $\partial \Omega$, in order to avoid technical issues that are irrelevant for our purposes. More generally, $\Om$ could also be such a domain in a Riemannian manifold. 

Let first $1<p<\infty$. For $u$ in the Sobolev space $W^{1,p}(\Omega)$, we may consider the functional 
 \begin{equation}
   \label{euc1}
   I_p(u)=\int_\Omega |\nabla u|^p dx.
 \end{equation}
Its Euler-Lagrange operator is the $p$-Laplacian
\begin{equation}
  \label{euc2}
  \Delta_p u= -\diver (|\nabla u|^{p-2}\nabla u)\ ;
\end{equation}
for $p=2$, we have, of course, the standard Laplace operator. 
Note that we use the $-$ sign in \eqref{euc2} both to make the operator a positive one and to conform to the conventions used in this paper. The eigenvalue problem arises when we look for critical points of $I_p$ under the constraint
\begin{equation}
  \label{euc3}
  \int_\Omega |u|^p dx =1,
\end{equation}
or equivalently, if we seek critical points of the Rayleigh quotient
\begin{equation}
  \label{euc4}
\frac{\int_\Omega |\nabla u|^p dx}{\int_\Omega |u|^p dx}
\end{equation}
among functions $u \nequiv 0$. To make the problem well formulated, we need to impose a boundary condition, and we consider here the Dirichlet condition
\begin{equation}
  \label{euc5}
  u \equiv 0 \text{ on } \partial \Omega.
\end{equation}
On a compact Riemannian manifold $M$ with boundary $\partial M$, we can do the same when we integrate in \eqref{euc1}, \eqref{euc3} with respect to the Riemannian volume measure, and let $\nabla$ and $\diver$ denote the Riemannian gradient and divergence operators. When $\partial M=\emptyset$, we do not need to impose a boundary condition. 

Eigenfunctions and eigenvalues then have to satisfy the equation
\begin{equation}
  \label{euc6}
  \Delta_p u= \lambda |u|^{p-2}u. 
\end{equation}
For $1<p <\infty$, the functionals in \eqref{euc1} and \eqref{euc3} are strictly convex, and the spectral theory is similar to that for $p=2$, that is, the case of the ordinary Laplacian, which is a well studied subject. (See for instance \cite{Valtorta14} for the situation on a Riemannian manifold.) For $p=1$, however, the functionals are no longer strictly convex, and things get more complicated. \rf{euc6} then formally becomes
\begin{equation}
  \label{euc7}
  -\diver \left( \frac{\nabla u}{|\nabla u|}\right)=\lambda \frac{u}{|u|}.
\end{equation}
In \eqref{euc6} for $p>1$, we may put the right hand side $=0$ at points where $u=0$, but this is no longer possible in \eqref{euc7}. This eigenvalue problem has been studied by Kawohl, Schuricht and their students and collaborators, as well as by Chang, and we shall summarize their results. Some references are \cite{Kawohl03,Schuricht06,Kawohl07,Chang2,Milbers10,Milbers13,Parini11,Littig14,Kawohl15,Lucia20}. 

 One therefore formally replaces \eqref{euc7} by defining substitute $z$ of $\frac{\nabla u}{|\nabla u|}$ and a substitute $s$ of $\frac{u}{|u|}$, leading to 
\begin{equation}
  \label{euc8}
  -\diver z=\lambda s
\end{equation}
where $s\in L^\infty(\Omega)$ satisfies
\begin{equation}
  \label{euc9}
  s(x)\in \mathrm{Sgn}(u(x))
\end{equation}
with
 $$\mathrm{Sgn}(t):=\begin{cases}
 \{1\} & \text{if } t>0,\\
 [-1,1] & \text{if }t=0,\\
 \{-1\} & \text{if }t<0,
 \end{cases}
$$
and the vector field $z\in L^{\infty}(\Omega,\R^n)$ satisfies
\begin{equation}
  \label{euc10}
  \|z\|_\infty=1,\quad \diver z \in L^n(\Omega),\quad  -\int_\Omega u\ \diver z dx= I(u)
\end{equation}
where
\begin{equation}
  \label{euc11}
  I(u)=\int_\Omega |Du|dx + \int_{\partial \Omega} |u^{\partial \Omega}|d\mathcal{H}^{n-1}.
\end{equation}
Again \eqref{euc11} needs some explanation. In fact, while for $p>1$, the natural space to work in is $W^{1,p}(\Omega)$, for $p=1$, it is no longer $W^{1,1}(\Omega)$, but rather $BV(\Omega)$. This space (for a short introduction, see for instance \cite{Jost98}) consists of all functions $L^1(\Omega)$ for which
\begin{equation}
  \label{euc12}
  |Du|(\Om)=\sup\left\{\int_\Om u\diver gdx:g=(g^1,\ldots g^n)\in C_0^\infty(\Om,\R^n),
  |g(x)|\le1\text{ for all }  x\in\Om\right\}<\infty
\end{equation}
Note that when $u \in C^1(\Om)$, we have 
$$ \int_\Om u\diver gdx= -\int_\Om \sum_i g^i \frac{\partial u}{\partial x^i} dx,$$
and thus, $BV$-functions permit such an integration by parts in a weak sense. More precisely, for a $BV$-function $u$, its distributional gradient is represented by a finite $\R^n$ valued signed measure $|Du|dx$, and we can write
\begin{equation}
  \label{euc13}
\int_\Om u\diver gdx= -\int_\Om g|Du|dx \text{ for } g\in C_0^\infty(\Om,\R^n).
\end{equation}
Also, $u\in BV(\Om)$ has a well-defined trace $u^{\partial \Omega}\in L^1(\partial \Om)$, and \eqref{euc13} generalizes to
\begin{equation}
  \label{euc14}
\int_\Om u\diver h dx= -\int_\Om h|Du|dx +\int_{\partial \Om}u^{\partial \Omega} (h\nu)d\mathcal{H}^{n-1} \text{ for } h\in C^1(\Om,\R^n)\cap C(\overline{\Om},\R^n)
\end{equation}
where $\nu$ is the outer unit normal of $\partial \Om$. \\
Importantly, $BV$-functions can be discontinuous along hypersurfaces. A Borel set $E\subset \Om$ has finite perimeter if its characteristic function $\chi_E$ satisfies
$$ |D\chi_E|(\Om) \Big( =  \sup\left\{\int_E\diver g:
       g\in C_0^\infty(\Om,\R^n), |g|\le1\right\}\Big)<\infty .$$
For instance, if the boundary of $E$ is a  compact Lipschitz hypersurface, then the perimeter of $E$ is simply the Hausdorff measure $\mathcal{H}^{n-1}(\partial E)$. And if $E\subset \Om$, we have
\begin{equation}
  \label{euc15}
 |D\chi_E|:=|D\chi_E|(\R^n)= |D\chi_E|(\Om) + \mathcal{H}^{n-1}(\partial E \cap \partial \Om).
\end{equation}
The problem with \eqref{euc8}, however, is that in general it has too many solutions, as it becomes rather arbitrary on sets of positive measure where $u$ vanishes, see \cite{Milbers13}. The solutions that one is really interested in should be the critical points of a variational principle, with the vanishing of the weak slope of \cite{Liusternik} as the appropriate criterion. Inner variations provide another necessary criterion \cite{Milbers13}. Viscosity solutions provide  another criterion which, however, is still not stringent enough \cite{Kawohl15}.

The Cheeger constant of $\Om$ then is defined as
\begin{equation}
  \label{euc16}
  h_1(\Om):= \inf_{E\subset \overline{\Om}}\frac{|D\chi_E|}{|E|}
\end{equation}
where $|E|$ is the Lebesgue measure of $E$. A set realizing the infimum in \eqref{euc16} is called a Cheeger set, and every bounded Lipschitz domain $\Omega$ possesses at least one Cheeger set. For such a Cheeger set $E\subset \Om$, $\partial E \cap \Om$ is smooth except possibly for a singular set of Hausdorff dimension at most $n-8$ and of constant mean curvature $\frac{1}{n-1}h_1(\Om)$ at all regular points. When $\Om$ is not convex, its Cheeger set need not be unique. 

In fact, $h_1(\Om)$ equals the first eigenvalue of the 1-Laplacian. More precisely,
\begin{equation}
  \label{euc17}
  h_1(\Om)= \inf_{u\in BV(\Om), u\not\equiv 0}\frac{\int_\Omega |Du|dx + \int_{\partial \Omega} |u^{\partial \Omega}|d\mathcal{H}^{n-1}}{\int_\Om |u|dx}=:\lambda_{1,1}(\Om)
\end{equation}
is the smallest $\lambda\neq 0$ for which there is a nontrivial solution $u$ of \eqref{euc7}, and such a $u$ is of the form $\chi_E$ for a Cheeger set, up to a multiplicative factor, of course. Also, if $\lambda_{1,p}(\Om)$ 
denotes the smallest nonzero eigenvalue of \eqref{euc6}, then
\begin{equation}
  \label{euc18}
  \lim_{p\to 1^+}\lambda_{1,p}(\Om)=\lambda_{1,1}(\Om).
\end{equation}
We also have the lower bound
\begin{equation}
  \label{euc19}
 \lambda_{1,p}(\Om)\ge \left(\frac{h_1(\Om)}{p}\right)^p 
\end{equation}
generalizing the original Cheeger bound for $p=2$.

More generally, for any family of eigenvalues $\lambda_{k,p}(\Om)$ of \eqref{euc6}, $\lim_{p\to 1^+}\lambda_{k,p}(\Om)$ is an eigenvalue of \eqref{euc7}. The converse is not true, however; \eqref{euc7} may have more solutions than can be obtained as limits of solutions of \eqref{euc6}. \\

The functional $|Du|$ appears also in image denoising, in so-called TV models (where the acronym TV refers to the fact that $|Du|(\Om)$ is the total variation of the measure $|Du|dx$) introduced in \cite{Rudin92}. There, one wants to denoise a function $f:\Omega \to \R$ by smoothing it, and in the TV models, one wants to minimize a functional of the form
\begin{equation}
  \label{euc20}
  \int_\Om |Du| dx + \mu \int_\Om |u-f|dx. 
\end{equation}
$\int_\Om |u-f|$ is the so-called fidelity term that controls the deviation of the denoised version $u$ from the given data $f$. $\mu >0$ is a parameter that balances the smoothness and the fidelity term.  Formally, a minimizer $u$ has to satisfy an equation of the form
\begin{equation}
  \label{euc21}
 \diver \left( \frac{D u}{|D u|}\right)=\mu \frac{u-f}{|u-f|} 
\end{equation}
which is similar to \eqref{euc7}. It turns out, however, that when such a model is applied to actual data, the performance is not so good, and it has been found preferable to modify \eqref{euc20} to what is called a nonlocal model in image processing \cite{Gilboa07}. In \cite{Jin15}, such a model was derived from geometric considerations, and this may also provide some insight into the relation with the discrete models considered in this paper, we now recall the construction of that reference. \\
Let $\Om$ be a domain in $\R^n$ or some more abstract space, and  $\om:\Om \times\Om \to \R$  a nonnegative, symmetric function. $\om(x,y)$ can be interpreted as some kind of edge weight between the points $x, y$ for any pair $(x,y)\in \Omega\times \Om $. Here $x,y$ can also stand for patches in the image, and in our setting, they could also be vertices in a graph (in which case the integrals below would become sums). We define the average $\baro:\Om \to \R$ of $\om$  by
\[
\baro(x)= \int_\Omega \om(x,y) dy
\]
and assume that $\baro$ is positive almost everywhere. On a graph, while $\om$ is an edge function, $\baro$ would be a vertex function, $\baro(x)$ being the degree of the vertex $x$  with edge weights $\om(x,y)$. 
We first use $\baro(x)$ and $\om(x,y)$ to define the $L^2$-norms for functions $u:\Om\to\R$ and
 vector fields $p$, that is,  $p:\Om\times \Om \to \R$, 
 \begin{eqnarray*}
( u_1, u_2)_{L^2} &:=&\int_\Omega u_1 (x) u_2(x)\baro (x) dx\\
( p_1,p_2 )_{L^2} &:=&\int_{\Omega\times\Omega} p_1(x,y) p_2(x,y)\om(x,y) dxdy
 \end{eqnarray*}
 and the corresponding norms $|u|$ and $|p|$. \\
The discrete derivative of a function (an image)  $u:\Om\to \R$ is defined by
 \begin{equation}\label{euc22}
 D u (x,y)=u(y)-u(x).
 \end{equation}
Even though $Du$ does not depend on $\om$, it is in some sense analogous to a gradient, as we shall see below. 
Its pointwise norm then is given 
\begin{equation}\label{euc23}
|D u|(x) =\left(\frac 1{\baro (x)} \int_\Omega(u(y)-u(x))^2 \om(x,y) dy\right)^{\frac 12}.
\end{equation}
The divergence of a vector field $p:\Om \times\Om \to \R$   is  defined by 
\begin{equation}\label{euc24}
\diver p(x):= \frac 1{\baro(x)} \int_\Omega (p(x,y)-p(y,x)) \om(x,y) dy.
\end{equation}
Note that, in contrast to $Du$ for a function $u$, the divergence
of a vector field depends on the weight $\om$. For  $u:\Om\to \R$ and $p:\Om\times\Om\to \R$, we then have
\begin{equation}
  \label{euc25}
   ( Du, p)_{L^2}= - (  u, \diver p)_{L^2}, 
\end{equation} 
the analog of \eqref{euc13}.

With the  vector field $Du $ and the divergence operator $\diver $, we can define a Laplacian for functions
\begin{equation}\label{euc26} \Delta u(x):=
 -\diver (D u)= u (x)- \frac 1{\baro(x)} \int_\Omega u(y) \om(x,y) dy\ 
   ,
\end{equation}
which in the case of a graph is the Laplacian we have been using. 
The nonlocal TV (or BV) functional of \cite{Jin15} then is
\begin{equation}\begin{array}{rrl}
TV_{\om}(u)&:=&\int_\Omega |D u| \bar \om (x) dx \\
&=& \int_\Omega (\int_\Omega (u(y)-u(x))^2 \om (x,y) d y)^\frac 12 \sqrt {\baro(x)} dx.
\end{array}
\end{equation}
This leads to the nonlocal
TV model
\begin{equation}\label{euc27}
\begin{array}{rcl}
ROF_{\om}(u)&=& TV_{\om}(u)+\mu \int_\Omega |u-f| \baro (x) dx.\\
&=&\int_\Omega(\int _\Omega(u(y)-u(x))^2 \om (x,y) dy)^\frac 12 \sqrt {\baro (x)} dx +\mu
 \int_\Omega\int_\Omega |u(x)-f(x)|\om(x,y) dxdy.
 \end{array} 
 \end{equation}
It should be of interest to explore such models on hypergraphs. That would offer the possibility to account not only for correlations between pairs, but also between selected larger sets of vertices, for instance three collinear ones. 
\subsection{Basic notions on hypergraphs}\label{section:hyp}
\begin{definition}[\cite{Shi92}]
	An \textbf{oriented hypergraph} is a pair $\Gamma=(V,H)$ such that $V$ is a finite set of vertices and $H$ is a set such that every element $h$ in $H$ is a pair of disjoint elements $(h_{in},h_{out})$ (input and output) in $\mathcal{P}(V)\setminus\{\emptyset\}$. The elements of $H$ are called the \textbf{oriented hyperedges}. Changing the orientation of a hyperedge $h$ means exchanging its input and output, leading to the pair $(h_{out},h_{in})$.
\end{definition}With a little abuse of notation, we shall see $h$ as $h_{in}\cup h_{out}$.
\begin{definition}[\cite{MulasZhang}]
 Given $h\in H$, we say that two vertices $i$ and $j$ are \textbf{co-oriented} in $h$ if they belong to the same orientation sets of $h$; we say that they are \textbf{anti-oriented} in $h$ if they belong to different orientation sets of $h$.
\end{definition}

\begin{definition}
 Given $i\in V$, we say that two hyperedges $h$ and $h'$ contain $i$ with the \textbf{same orientation} if $i\in (h_{in}\cap h'_{in})\cup (h_{out}\cap h'_{out})$; we say that they contain $i$ with \textbf{opposite orientation} if $i\in (h_{in}\cap h'_{out})\cup (h_{out}\cap h'_{in})$.
\end{definition}

\begin{definition}[\cite{Sharp}]
The \textbf{degree} of a vertex $i$ is
	\begin{equation*}
	    \deg(i):=\#\text{ hyperedges containing $i$ only as an input or only as an output }
	\end{equation*}and the \textbf{cardinality} of a hyperedge $h$ is
	\begin{equation*}
	    \# h:=\#\{(h_{in}\setminus h_{out})\cup (h_{out}\setminus h_{in})\}=\#\text{ vertices in $h$ that are either only an input or only an output}.
	\end{equation*}
\end{definition}

From now on, we fix such an oriented hypergraph $\Gamma=(V,H)$ on $n$ vertices $1,\ldots,n$ and $m$ hyperedges $h_1,\ldots, h_m$. We assume that there are no vertices of degree zero. We denote by $C(V)$ the space of functions $f:V\rightarrow\R$ and we denote by $C(H)$ the space of functions $\gamma:H\rightarrow\R$.

\section{p-Laplacians for $p>1$}\label{section:pLaplacians}
\begin{definition}\label{def:vertexL}
 Given $p\in\R_{> 1}$, the \textbf{(normalized) vertex $p$-Laplacian} is $\Delta_p:C(V)\to C(V)$, where
 \begin{small}
\begin{equation*}
 \Delta_p f(i):=\frac{1}{\deg(i)}   \sum_{h\ni i}\left|\sum_{i'\text{ input of }h}f(i')-\sum_{i''\text{ output of }h}f(i'')\right|^{p-2}\left(\sum_{j\in h,o_h(i,j)=-1}f(j)-\sum_{j'\in h,o_h(i,j')=1}f(j')\right)
  \end{equation*}\end{small}
  where $$o_h(i,j)=\begin{cases}
  -1,&\text{ if }i,j\in h, i\text{ and }j \text{ are co-oriented in }h\\
  1,&\text{ if }i,j\in h, i\text{ and }j \text{ are anti-oriented in }h\\
  0,&\text{ otherwise.}
  \end{cases}$$
 We define its \textbf{eigenvalue problem} as
  \begin{equation}\label{eq:p-eigenvalue} \Delta_p f=\lambda |f|^{p-2}f.
  \end{equation}We say that a nonzero function $f$ and real number $\lambda$ satisfying \eqref{eq:p-eigenvalue} are an eigenfunction and the corresponding eigenvalue for $\Delta_p$.
\end{definition}
\begin{remark}Definition \ref{def:vertexL} generalizes both the graph $p$-Laplacian and the normalized Laplacian defined in \cite{Hypergraphs} for hypergraphs, which corresponds to the case $p=2$.
\end{remark}

\begin{remark}
The $p$-Laplace operators for classical hypergraphs that were introduced in \cite{hein2013total} coincide with the vertex $p$-Laplacians that we introduced here in the case of simple graphs, but not in the more general case of hypergraphs. In fact, the Laplacians in \cite{hein2013total} are related to the Lov\'asz extension, while the operators that we consider here are defined via the incidence matrix. Also, the corresponding functionals for the $p$-Laplacians in \cite{hein2013total} are of the form
\begin{equation*}
    f\mapsto \sum_{h\in H}\max\limits_{i,j\in h}\left|f(i)-f(j)\right|^p,
\end{equation*}
 and these are non-smooth in general, even for $p>1$. In our case, the corresponding functionals are of the form
 \begin{equation*}
     f\mapsto \sum_{h\in H}\left|\sum_{i\in h_{in}} f(i)-\sum_{j\in h_{out}} f(j)\right|^p,
 \end{equation*}
 and these are smooth for $p>1$. 
\end{remark}

\begin{definition}\label{def:hyperedgeL}
 Given $p\in\R_{> 1}$, the \textbf{(normalized) hyperedge $p$-Laplacian} is $\Delta^H_p:C(H)\to C(H)$, where
\begin{small}
\begin{equation*}
\Delta_p^H\gamma(h):=\sum_{i\in h}\frac{1}{\deg (i)}\left|\sum_{h'\ni i\text{ as input }}\gamma(h')-\sum_{h''\ni i\text{ as output }}\gamma(h'')\right|^{p-2}\Biggl(\sum_{h'\ni i, o_i(h,h')=-1}\gamma(h')-\sum_{h''\ni i, o_i(h,h'')=1}\gamma(h'')\Biggr)
\end{equation*}\end{small}

where 
$$o_i(h,h')=\begin{cases}
  -1,&\text{ if }h,h'\ni i \text{ with the same orientation }\\
  1,&\text{ if }h,h'\ni i \text{ with opposite orientation }\\
  0,&\text{ otherwise.}
  \end{cases}$$
 We define its \textbf{eigenvalue problem} as
  \begin{equation}\label{eq:p-hypereigenvalue} 
  \Delta^H_p \gamma=\lambda |\gamma|^{p-2}\gamma.
  \end{equation}We say that a nonzero function $\gamma$ and a real number $\lambda$ satisfying \eqref{eq:p-hypereigenvalue} are an eigenfunction and the corresponding eigenvalue for $\Delta^H_p$.
\end{definition}
\begin{remark}
    For $p=2$, Definition \ref{def:hyperedgeL} coincides with the one in \cite{Hypergraphs}. Also, as we shall see, while it is known that the nonzero eigenvalues of $\Delta_p$ and $\Delta_p^H$ coincide for $p=2$, this is no longer true for a general $p$.
\end{remark}

\subsection{Generalized min-max principle}\label{section:minmax}
For $p=2$, the Courant-Fischer-Weyl min-max principle can be applied in order to have a characterizations of the eigenvalues of $\Delta_2$ and $\Delta_2^H$ in terms of the \emph{Rayleigh Quotients} of the functions $f\in C(V)$ and $\gamma\in C(H)$, respectively, as shown in \cite{Hypergraphs}. In this section we prove that, for $p>1$, a generalized version of the min-max principle can be applied in order to know more about the eigenvalues of $\Delta_p$ and $\Delta_p^H$. Similar results are already known for graphs, as shown for instance in \cite{Minmax}. Before stating the main results of this section, we define the \emph{generalized Rayleigh Quotients} for functions on the vertex set and for functions on the hyperedge set.

\begin{definition}\label{def:rq}
 Let $p\in\R_{\geq 1}$. Given $f\in C(V)$, its \textbf{generalized Rayleigh Quotient} is
 \begin{equation*}
    \RQ_p(f):=\frac{\sum_{h\in H}\left|\sum_{i\text{ input of }h}f(i)-\sum_{j\text{ output of }h}f(j)\right|^p}{\sum_{i\in V}\deg(i)|f(i)|^p}.
\end{equation*}Analogously, the \textbf{generalized Rayleigh Quotient} of $\gamma\in C(H)$ is
\begin{equation*}
    \RQ_p(\gamma):=\frac{\sum_{i\in V}\frac{1}{\deg (i)}\cdot \biggl|\sum_{h': v\text{ input}}\gamma(h')-\sum_{h'': v\text{ output}}\gamma(h'')\biggr|^p}{\sum_{h\in H}|\gamma(h)|^p}.
\end{equation*}
 \end{definition}
 \begin{remark}
    It is clear from the definition of $\RQ_p(f)$ and $\RQ_p(\gamma)$ that
    \begin{equation*}
        \RQ_{\hat{p}}(f)=0\text{ for some }\hat{p}\,\iff\,\RQ_p(f)=0\text{ for all }p
    \end{equation*}and 
    \begin{equation*}
        \RQ_{\hat{p}}(\gamma)=0\text{ for some }\hat{p}\,\iff\,\RQ_p(\gamma)=0\text{ for all }p.
    \end{equation*}
\end{remark}
\begin{theorem}\label{theo:eigenvaluesRQ}
Let $p\in\mathbb{R}_{>1}$. $f\in C(V)\setminus \{0\}$ is an eigenfunction for $\Delta_p$ with corresponding eigenvalue $\lambda$ if and only if $$\nabla \RQ_p(f)=0 \text{ and }\lambda=\RQ_p(f).$$ Similarly, $\gamma\in C(H)\setminus \{0\}$ is an eigenfunction for $\Delta^H_p$ with corresponding eigenvalue $\mu$ if and only if $$\nabla \RQ_p(\gamma)=0 \text{ and }\lambda=\RQ_p(\gamma).$$
\end{theorem}
\begin{proof}
For $p\in\mathbb{R}_{>1}$, $\RQ_p$ is differentiable on $\mathbb{R}^n\setminus 0$. Also, \begin{align*}
 \partial_i\RQ_p(f)&= \partial_i\left(\frac{\sum_{h\in H}\left|\sum_{i\text{ input of }h}f(i)-\sum_{j\text{ output of }h}f(j)\right|^p}{\sum_{i\in V}\deg(i)|f(i)|^p}\right)
\\&=\frac{\partial_i\left(\sum\limits_{h}\left|\sum\limits_{i\in h_{in}}f(i)-\sum\limits_{j\in h_{out}}f(j)\right|^p\right)-\RQ_p(f)\cdot \partial_i\left(\sum\limits_{i}\deg(i)|f(i)|^p\right)}{\sum_{i}\deg(i)|f(i)|^p}
\\&=\frac{p\cdot\deg(i)\cdot\Delta_pf(i)-\RQ_p(f)\cdot p\cdot\deg(i)\cdot|f(i)|^{p-2}f(i)}{\sum_{i}\deg(i)|f(i)|^p}
\\&= p\cdot \deg(i)\cdot \frac{ \Delta_p f(i)-\RQ_p(f) |f(i)|^{p-2}f(i)}{\sum_{i=1}^n \deg(i)|f(i)|^p}, \end{align*}
      where we have used the fact that
    $$\partial_t|t|^p=p|t|^{p-1}\mathrm{sign}(t)=p|t|^{p-2}t.$$
  Hence,
  \begin{align*}
      \nabla \RQ_p(f)=0&\iff \partial_i\RQ_p(f)=0~\forall i 
      \\& \iff \Delta_p f=\RQ_p(f) |f|^{p-2}f
      \\& \iff f \text{ is an eigenfunction for }\Delta_p \text{ with eigenvalue }\RQ_p(f).
  \end{align*}Furthermore, if $f'$ is an eigenfunction corresponding to any eigenvalue $\lambda$, then $\Delta_p f=\lambda |f|^{p-2}f$, therefore 
    $$\langle\Delta_p f,f\rangle=\langle\lambda |f|^{p-2}f,f\rangle$$
    which can be simplified as $$\RQ_p(f)=\lambda.$$ This proves the claim for $\Delta_p$. The case of $\Delta^H_p$ is similar. We have that
    \begin{align*}
 \partial_h\RQ_p(\gamma)&= \partial_h \left(\frac{\sum_{i\in V}\frac{1}{\deg (i)}\cdot \biggl|\sum_{h': v\text{ input}}\gamma(h')-\sum_{h'': v\text{ output}}\gamma(h'')\biggr|^p}{\sum_{\hat{h}\in H}|\gamma(\hat{h})|^p} \right)\\
 &= \frac{\partial_h\left(\sum_{i\in V}\frac{1}{\deg (i)}\cdot \biggl|\sum_{h': v\text{ input}}\gamma(h')-\sum_{h'': v\text{ output}}\gamma(h'')\biggr|^p\right)-\RQ_p(\gamma)\cdot \partial_h\left(\sum_{\hat{h}\in H}|\gamma(\hat{h})|^p\right)}{\sum_{\hat{h}\in H}|\gamma(\hat{h})|^p} \\
 &=\frac{p\cdot \Delta_p^H\gamma(h)-\RQ_p(\gamma)\cdot p\cdot \left(|\gamma(h)|^{p-2}\gamma(h)\right)}{\sum_{\hat{h}\in H}|\gamma(\hat{h})|^p}.
 \end{align*}Therefore,
\begin{align*}
    \nabla \RQ_p(\gamma)=0&\iff  \partial_h\RQ_p(\gamma)=0~\forall h 
    \\&\iff \Delta_p^H\gamma=\RQ_p(\gamma)|\gamma|^{p-2}\gamma
    \\&\iff \gamma \text{ is an eigenfunction for }\Delta^H_p \text{ with eigenvalue }\RQ_p(\gamma).
\end{align*}This proves the first implication for $\Delta^H_p$. The inverse implication is analogous to the case of $\Delta_p$.
\end{proof}
\begin{corollary}\label{cor:minmaxRQ}For all $p>1$,
\begin{equation}\label{eq:mineigenvalue}
    \min_{f\in C(V)}\RQ_p(f) \qquad (\text{resp.} \max_{f\in C(V)}\RQ_p(f))
\end{equation}is the smallest (resp. largest) eigenvalue of $\Delta_p$, and $f$ realizing \eqref{eq:mineigenvalue} is a corresponding eigenfunction. \newline

Analogously,
\begin{equation}\label{eq:mineigenvalueh}
    \min_{\gamma\in C(H)}\RQ_p(\gamma) \qquad (\text{resp.} \max_{\gamma\in C(H)}\RQ_p(\gamma))
\end{equation}is the smallest (resp. largest) eigenvalue of $\Delta^H_p$, and $\gamma$ realizing \eqref{eq:mineigenvalueh} is a corresponding eigenfunction.
\end{corollary}
\begin{proof}
  By Fermat's theorem, if $f\ne 0$ minimizes or maximizes $\RQ_p$ over $\mathbb{R}^n\setminus 0$, then $\nabla \RQ_p(f)=0$. The claim for $\Delta_p$ then follows by Theorem \ref{theo:eigenvaluesRQ}, and the case of $\Delta^H_p$ is analogous.
\end{proof}

We now give a preliminary definition, before stating the generalized min-max principle.

\begin{definition}
For a centrally symmetric set $S$ in $\mathbb{R}^n$, its \textbf{Krasnoselskii $\mathbb{Z}_2$ genus} is defined as
\begin{equation*}
\mathrm{gen}(S) :=
\begin{cases}
\min\limits\{k\in\mathbb{Z}^+: \exists\; \text{odd continuous}\; h: S\setminus 0\to \mathbb{S}^{k-1}\} & \text{if}\; S\setminus 0\ne\emptyset,\\
0 & \text{if}\; S\setminus 0=\emptyset.
\end{cases}
\end{equation*}
For each $k\geq 1$, we let $\mathrm{Gen}_k:=\{ S\subset \mathbb{R}^n: S\text{ centrally symmetric with } \mathrm{gen}(S)\ge k\}$.
\end{definition}
\begin{remark}
    From the above definition we get an inclusion chain $$\mathrm{Gen}_1\supset \mathrm{Gen}_2\supset\ldots\mathrm{Gen}_n\supset \emptyset= \mathrm{Gen}_{n+1}=\ldots=\emptyset.$$ Therefore, the  Krasnoselskii $\mathbb{Z}_2$ genus gives a graded index of the family of all centrally symmetric sets with center at $0$ in $\mathrm{R}^n$, which generalizes the (linear) dimension of subspaces.  
\end{remark}

\begin{theorem}[Generalized min-max principle]\label{theo:minmax}
Let $p\in \mathbb{R}_{>1}$. For $k=1,\ldots,n$, the constants
\begin{equation}\label{eq:min-max-Z2-eigenvalue}
    \lambda_{k}(\Delta_p) := \inf_{S\in\mathrm{Gen}_k}\sup\limits_{f\in S\setminus 0} \RQ_p(f)
\end{equation}
are eigenvalues of $\Delta_p$. They satisfy $$\lambda_1\le \ldots\leq \lambda_n$$ and, if $\lambda=\lambda_{k+1}=\ldots=\lambda_{k+l}$ for $0\le k<k+l\le n,$ then $$\mathrm{gen}(\{\text{eigenfunctions corresponding to }\lambda\})\ge l.$$
The same  holds for the constants 
\begin{equation}\label{eq:min-max-Z2-eigenvalueh}
    \mu_{k}(\Delta^H_p) := \inf_{S\in\mathrm{Gen}_k}\sup\limits_{f\in S\setminus 0} \RQ_p(\gamma), \qquad k=1,\ldots,m,
\end{equation}that are eigenvalues of $\Delta^H_p$.
\end{theorem}
\begin{proof}By Theorem \ref{theo:eigenvaluesRQ}, in order to prove the claim for $\Delta_p$ it suffices to show that $\lambda_{k}(\Delta_p)$ defined in 
\eqref{eq:min-max-Z2-eigenvalue} is a critical value of $\RQ_p$. Let
\begin{equation*}
    \|f\|_p:=\left(\sum_{i\in V}\deg(i)|f(i)|^p\right)^{\frac1p}
\end{equation*}be the $p$-norm with weights given by the degrees, and let
\begin{equation*}
     E_p(f):=\sum_{h\in H} \left|\sum_{j\in h_{in}} f(j)-\sum_{j'\in h_{out}} f(j')\right|^p.
\end{equation*}Then, $\RQ_p(f)=E_p(\frac{f}{\|f\|_p})$. Now, consider the $l^p$-sphere $S_p=\{f\in\mathbb{R}^n:\|f\|_p=1\}$. We have that
\begin{equation*}
    \sup\limits_{f\in S\setminus 0}\RQ_p(f)=\sup\limits_{f\in \mathrm{R}_+S\setminus 0}\RQ_p(f)=\sup\limits_{f\in \mathrm{R}_+S\cap S_p}E_p(f),
\end{equation*}where $\mathrm{R}_+S:=\{cg:g\in S,c>0\}$. Therefore, it can be verified that  $$\lambda_k(\Delta_p)=\inf\limits_{S\subset S_p,S\in \mathrm{Gen}_k,f\in S}E_p(f).$$From the Liusternik-Schnirelmann Theorem applied to the smooth function $E_p$ restricted to the smooth $l^p$-sphere $S_p$ it follows that such a min-max quantity must be an eigenvalue of $E_p$ on $S_p$. This proves the claim for $\Delta_p$. The case of $\Delta_p^H$ is similar, if we consider
\begin{equation*}
    \|\gamma\|_p:=\left(\sum_{h\in H}|\gamma(h)|^p\right)^{\frac1p},
\end{equation*}
\begin{equation*}
    E_p(\gamma):=\sum_{i\in V}\frac{1}{\deg (i)}\left|\sum_{h'\ni i\text{ as input }}\gamma(h')-\sum_{h''\ni i\text{ as output }}\gamma(h'')\right|^p
\end{equation*}and $S_p:=\{\gamma\in \mathbb{R}^m:\|\gamma\|_p=1\}$.
\end{proof}
\begin{remark}
    For the case of $p=2$, a linear subspace $X$ in $\mathbb{R}^n$ with $\dim X=k$ satisfies $\mathrm{gen}(X)=k$ and by considering the sub-family $$\widetilde{\mathrm{Gen}}_k:=\{\text{linear subspace with dimension at least }k\}\subset \mathrm{Gen}_k$$ we have $$\lambda_{k}(\Delta_2) = \inf\limits_{S\in\mathrm{Gen}_k}\sup\limits_{f\in S\setminus 0} \RQ_2(f)=\inf\limits_{S\in\widetilde{\mathrm{Gen}}_k}\sup\limits_{f\in S\setminus 0} \RQ_2(f).$$
   This coincides with the Courant-Fischer-Weyl min-max principle. On the other hand, for $p>1$, we only know that $$\lambda_{k}(\Delta_p) = \inf\limits_{S\in\mathrm{Gen}_k}\sup\limits_{f\in S\setminus 0} \RQ_p(f)\le\inf\limits_{S\in\widetilde{\mathrm{Gen}}_k}\sup\limits_{f\in S\setminus 0} \RQ_p(f).$$ 
   In particular, while for $p=2$ we know that the $n$ eigenvalues of $\Delta_p$ (resp. the $m$ eigenvalues of $\Delta_p^H$ appearing in Theorem \ref{theo:minmax}) are \emph{all the eigenvalues} of $\Delta_p$ (resp. $\Delta_p^H$), we don't know whether $\Delta_p$ and $\Delta_p^H$ have also more eigenvalues, for $p\ne2$. This is still an open question also for the graph case. In other words, we don't know whether all eigenvalues of $\Delta_p$ and $\Delta_p^H$ can be written in the min-max Rayleigh Quotient form. 
   \end{remark}
   
   \begin{conjecture}\label{conj-minmax}
   For $1<p<\infty$, all eigenvalues of $\Delta_p$ are min-max eigenvalues.
   \end{conjecture}
    
   We formulate this conjecture, because for the $p$-Laplacian on domains and manifolds as well as on graphs, it is an open problem whether all the eigenvalues of the $p$-Laplacian are of the min-max form (see \cite{DegiovanniMazzoleni19,CDG04,BindingRynne08} and \cite{Minmax}). Thus, as far as we know, Conjecture \ref{conj-minmax} is open in both the continuous and the discrete setting.

Throughout the paper, given $p>1$ we shall denote by
\begin{equation*}
    \lambda_1\leq\ldots\leq \lambda_n \qquad \text{and}\qquad \mu_1\leq\ldots\leq \mu_m
\end{equation*}the eigenvalues of $\Delta_p$ and $\Delta^H_p$, respectively, which are described in Theorem \ref{theo:minmax}. We shall call them the \textbf{min-max eigenvalues}. Note that, although we cannot say a priori whether these are \emph{all} the eigenvalues of the $p$-Laplacians, in view of Corollary \ref{cor:minmaxRQ} we can always say that 
\begin{equation*}
    \lambda_1=\min_{f\in C(V)}\RQ_p(f),\,\,\,\, \lambda_n=\max_{f\in C(V)}\RQ_p(f), \,\,\,\, \mu_1=\min_{\gamma\in C(H)}\RQ_p(\gamma), \,\,\,\, \mu_m=\max_{\gamma\in C(H)}\RQ_p(\gamma).
\end{equation*}

\section{1-Laplacians}\label{section:1Laplacians}
In this section we generalize the well known $1$-Laplacian for graphs \cite{Hein2,Chang,Chang2} to the case of hypergraphs.
\begin{definition}
  The \textbf{1-Laplacian} is the set-valued operator such that, given $f\in C(V)$,\begin{small}
  \begin{equation*}
 \Delta_1 f:=\left\{\sum_{i\in V}\frac{1}{\deg(i)}   \sum_{h\ni i} z_{ih}\vec e_i\left|z_{ih}\in \mathrm{Sgn}\left(\sum_{j\in h,o_h(i,j)=-1}f(j)-\sum_{j'\in h,o_h(i,j')=1}f(j')\right),z_{ih}=o_h(i,j)z_{jh}\right.\right\}
  \end{equation*}\end{small}
 where $\vec e_1,\ldots,\vec e_n$ is the orthonormal basis of $\R^n$ and $$\mathrm{Sgn}(t):=\begin{cases}
 \{1\} & \text{if } t>0,\\
 [-1,1] & \text{if }t=0,\\
 \{-1\} & \text{if }t<0.
 \end{cases}
$$
Analogously, the \textbf{hyperedge 1-Laplacian} for functions $\gamma\in C(H)$ is
\begin{small}\begin{equation*}
 \Delta^H_1 \gamma:=\left\{\sum_{h\in H}   \sum_{i\in h} \frac{1}{\deg (i)} z_{ih}\vec e_h\left|z_{ih}\in \mathrm{Sgn}\left(\sum_{h'\ni i, o_i(h,h')=-1}\gamma(h')-\sum_{h''\ni i, o_i(h,h'')=1}\gamma(h'')\right),z_{ih}=o_i(h,h')z_{ih'}\right.\right\}
  \end{equation*}\end{small}
 where $\vec e_{h_1},\ldots,\vec e_{h_m}$ is the orthonormal basis of $\R^m$.
\end{definition}

For any $f\in C(V)$, $\Delta_1 f$ is a compact convex set in $C(V)\cong\R^n$, as well as $$\mathrm{Sgn}(f):=\{g\in C(V): g(i)\in \mathrm{Sgn}(f(i)),\,\forall i\}.$$

\begin{remark}\label{rem:upper-continuous-p-Lap}
The 1-Laplacian is the limit of the $p$-Laplacian with respect to the set-valued upper limit, i.e.\
$$\Delta_1f= \limsup\limits_{p\to1^+,\delta\to 0^+}\Delta_p(\mathbb{B}_\delta(f)) = \lim\limits_{\delta\to 0^+}\lim\limits_{p\to1^+}\mathrm{conv}(\Delta_p(\mathbb{B}_\delta(f)))$$
where $\mathbb{B}_\delta(f)$ is the ball with radius $\delta$ and center $f$.  In other words, $\Delta_1f$ is the set of limit points of $\Delta_pf'$ when $p\to 1$ and $f'\to f$. On the one hand, if $f$ is such that $\sum_{i\in h_{in}}f(i)\ne \sum_{i\in h_{out}}f(i)$ for all $h\in H$, then $\Delta_1f=\lim\limits_{p\to1^+}\Delta_pf$ in the classical sense. On the other hand, for a general $f\in C(V)$, the limit may not exist. To some extent, the set-valued upper limit ensures the upper semi-continuity of the family of $p$-Laplacians, that is, the set-valued mapping  $[1,\infty)\times C(V)\ni (p,f)\mapsto \Delta_pf\in C(V)$ is upper semi-continuous.
\end{remark}

 \begin{definition}
The \textbf{eigenvalue problem} of $\Delta_1$ is to find the eigenpair $(\lambda,f)$ such that
$$
 \Delta_1 f\bigcap \lambda \mathrm{Sgn}(f)\ne\emptyset
$$
or equivalently, in terms of Minkowski summation,
$$0\in \Delta_1 f- \lambda \mathrm{Sgn}(f).$$
In coordinate form it means that there exist $$z_{ih}\in \mathrm{Sgn}\left(\sum_{j\in h,o_h(i,j)=-1}f(j)-\sum_{j'\in h,o_h(i,j')=1}f(j')\right)$$ with $z_{ih}=o_h(i,j)z_{jh}$ for $i,j\in h$, and $z_i\in \mathrm{Sgn}(f(i))$ such that
\begin{equation}\label{eq:coordinate-1-Laplacian-problem}
   \sum_{h\ni i}z_{ih}=\lambda \deg(i) z_i,\;\;\forall i\in V. 
\end{equation}
\end{definition}
\begin{remark}\label{remark:1-Lap}
A shorter coordinate form of the eigenvalue problem for the $1$-Laplacian is
\begin{align}
&\exists z_i\in \mathrm{Sgn}(f(i))\;\text{ and }\;z_{h}\in \mathrm{Sgn}\left(\sum_{i\in h_{in}}f(i)-\sum_{i\in h_{out}}f(i)\right)\;\text{ s.t.} \notag
\\&
 \sum_{h_{in}\ni i}z_{h}-\sum_{h_{out}\ni i}z_{h}=\lambda \deg(i) z_i,\;\;\forall i\in V. \label{eq:brief-1-Lap}
\end{align}Observe also that $(\sum_{i\in h_{in}}f(i)-\sum_{i\in h_{out}}f(i))z_h=|\sum_{i\in h_{in}}f(i)-\sum_{i\in h_{out}}f(i)|$ and $f(i)z_i=|f(i)|$, for all $h\in H$ and for all $i\in V$. 
\end{remark}

The eigenvalue problem of $\Delta_1^H$ can be defined in an analogous way. In particular, all results shown in this section for $\Delta_1$ also hold for $\Delta_1^H$. Without loss of generality, we only prove them for $\Delta_1$.
\begin{definition}
For the \emph{generalized Rayleigh Quotient} $RQ_1$ (cf. Definition \ref{def:rq}), its \textbf{Clarke derivative} at $f\in C(V)$ is
$$\nabla \RQ_1(f):=\left\{\xi\in C(V)\left|\limsup_{g\to f, t\to 0^+}\frac{\RQ_1(g+t\eta)-\RQ_1(g)}{t}\ge \langle \xi,\eta\rangle,\forall \eta\in C(V)\right.\right\}.$$
This is a compact convex set in $C(V)$.
\end{definition}
\begin{remark}
Clarke introduced such a derivative for locally Lipschitz functions, in the field of non\-smooth optimization \cite{Clarke,Clarke2}. Clearly, $\RQ_1$ is not smooth, but it is piecewise smooth (therefore locally Lipschitz) on $\R^n\setminus 0$. Hence, the Clarke derivative for $\RQ_1$ is well defined. Also, since the Clarke derivative coincides with the usual derivative for smooth functions, we choose to denote it by $\nabla$ also for locally Lipschitz functions.
\end{remark}
\begin{definition}
Given $f\in C(V)$, let $$E_1(f):=\sum_{h\in H}\left|\sum_{i\in h_{in}}f(i)-\sum_{i\in h_{out}}f(i)\right| \qquad\text{and}\qquad \|f\|_1:=\sum_{i\in V} \deg(i)|f(i)|.$$
\end{definition}
\begin{proposition}For all $i\in V$,
$$(\nabla E_1(f))(i)=\deg(i)\Delta_1f(i)\qquad\text{and}\qquad (\nabla\|f\|_1)(i)=\deg(i)\mathrm{Sgn}(f(i)).$$
\end{proposition}
\begin{proof}
Note that the Clarke derivative of the function $\mathbb{R}\ni t\mapsto |t|$ is $\mathrm{Sgn}(t)$.
Hence, by the chain rule in nonsmooth analysis, for $a_1,\ldots,a_k\in\R$, 
$$\nabla_{t_1,\ldots,t_k} |a_1t_1+\ldots+a_kt_k|=\left\{(a_1s,\ldots,a_ks)\in\R^k:s\in \mathrm{Sgn}(a_1t_1+\ldots+a_kt_k)\right\}.
$$
Finally, applying the additivity of Clarke's derivative, we derive the desired identities.  
\end{proof}

\begin{theorem}[Min-max principle for the $1$-Laplacian]\label{thm:1-Laplace}
If $f$ is a critical point of the function $\RQ_1$, i.e.\ $0\in \nabla \RQ_1(f)$, then $f$ is an eigenfunction and $\RQ_1(f)$ is the corresponding eigenvalue of $\Delta_1$. A function $f\in C(V)\setminus 0$ is a maximum (resp. minimum) eigenfunction of $\Delta_1$ if and only if it is a maximizer (resp. minimizer) of $\RQ_1$; $\lambda$ is the largest (resp. smallest) eigenvalue of $\Delta_1$ if and only if it is the maximum (resp. minimum) value of $\RQ_1$.\newline
Also, the constants
\begin{equation}\label{eq:min-max-Z2-eigenvalue2}
    \lambda_{k}(\Delta_1) := \inf_{S\in\mathrm{Gen}_k}\sup\limits_{f\in S\setminus 0} \RQ_1(f)
\end{equation}
are eigenvalues of $\Delta_1$. Furthermore, $\lim\limits_{p\to 1^+} \lambda_{k}(\Delta_p)= \lambda_{k}(\Delta_1)$, and any limit point of $\{f_{k,p}\}_{p>1}$ is an eigenfunction of $\Delta_1$ w.r.t. $\lambda_k(\Delta_1)$, where $f_{k,p}$ is an eigenfunction\footnote{For convenience, w.l.o.g. we normalize the eigenfunctions $f_{k,p}$ of $\lambda_k(\Delta_p)$, i.e., we assume $\|f_{k,p}\|_p=1$. } of  $\lambda_k(\Delta_p)$, $\forall k=1,\ldots,n$. Besides, if $\lim\limits_{p\to 1^+} \lambda_{k}(\Delta_p)=\lim\limits_{p\to 1^+} \lambda_{k+l}(\Delta_p)$ for some $k,l\in\mathbb{N}_+$, then $\lambda_k(\Delta_1)$ has the multiplicity at least $l+1$. 
\end{theorem}
\begin{proof}The proof is based on the theory of Clarke derivative, established in \cite{Clarke2}.\newline

Let $f$ be a critical point of the function $\RQ_1$. By the chain rule for the Clarke derivative, 
\begin{align*}
&0\in \nabla \RQ_1(f)\subset \frac{\nabla E_1(f)-\RQ_1(f)\nabla \|f\|_1}{\|f\|_1}
\\ \Longrightarrow~&0\in \nabla E_1(f)-\RQ_1(f)\nabla \|f\|_1
\\ \Longleftrightarrow~&0\in\Delta_1f-\RQ_1(f)\mathrm{Sgn}(f)f.
\end{align*}
Therefore, $f$ is an eigenfunction of $\Delta_1$, and $\RQ_1(f)$ is the corresponding eigenvalue. Also, again by the basic results on Clarke derivative, if $f$ is a maximizer (minimizer) of $\RQ_1$, then $0\in\nabla\RQ_1(f)$. Hence, $0\in\Delta_1f-\RQ_1(f)\mathrm{Sgn}(f)$. Thus, $f$ is an eigenfunction, and $\RQ_1(f)$ is a corresponding eigenvalue.\newline

Now, if $f$ is an eigenfunction corresponding to an eigenvalue $\lambda$, i.e.\ $0\in\Delta_1f-\lambda\mathrm{Sgn}(f)$ or equivalently 
\begin{equation}\label{eq:eigen-problem-E_1}
0\in \nabla E_1(f)-\lambda\nabla \|f\|_1,
\end{equation}
 then by the Euler identity for one-homogeneous Lipschitz functions, $$\langle g,f\rangle=E_1(f) \qquad \forall g\in \nabla E_1(f).$$  Therefore, by \eqref{eq:eigen-problem-E_1}, we get that
 $0=E_1(f)-\lambda \|f\|_1$, which implies $\lambda=\RQ_1(f)$. Hence, the maximum (resp. the minimum) of $\RQ_1$ is the largest (resp. smallest) eigenvalue of $\Delta_1$. \newline

The min-max principle \eqref{eq:min-max-Z2-eigenvalue2} is a consequence of the nonsmooth version of the Liusternik-Schnirelmann Theorem \cite{Liusternik}, and thus we omit the details of the proof.\newline

The convergence property $\lim\limits_{p\to 1^+} \lambda_{k}(\Delta_p)= \lambda_{k}(\Delta_1)$ is a consequence of the result on Gamma-convergence of minimax values \cite{Limitminmax}.\newline

Now, without loss of generality, we may assume that $f_{k,p}\to f_*$, $p\to1^+$. Then, according to Remark \ref{rem:upper-continuous-p-Lap}, $\lim\limits_{p\to1^+}\Delta_pf_{k,p}\in  \Delta_1 f_*$.  Similarly, $|f_{k,p}(i)|^{p-2}f_{k,p}(i)\to \mathrm{sign}(f_*(i))$  as $p$ tends to $1^+$. By taking $p\to1^+$ in the equality  $$0=\Delta_pf_{k,p}(i)-\lambda_{k}(\Delta_p)|f_{k,p}(i)|^{p-2}f_{k,p}(i),\;\;\;\forall i\in V$$ 
we get 
$$0=\lim\limits_{p\to1^+}\Delta_pf_{k,p}(i)-\lambda_{k}(\Delta_i)\mathrm{sign}(f_*(i))\in \Delta_1 f_*(i)-\lambda_{k}(\Delta_i)\mathrm{Sgn}(f_*(i)),\;\;\forall i\in V,$$ which means that $f_*$ is an eigenfunction of $\Delta_1$.\newline

The condition $\lim\limits_{p\to 1^+} \lambda_{k}(\Delta_p)=\lim\limits_{p\to 1^+} \lambda_{k+l}(\Delta_p)$ implies $\lambda_k(\Delta_1)=\lambda_{k+1}(\Delta_1)=\ldots=\lambda_{k+l}(\Delta_1)$, which derives that $\lambda_k(\Delta_1)$ has the multiplicity at least $(l+1)$ according to the Liusternik-Schnirelmann Theory. This completes the proof.


\end{proof}

Analogously to the case of $p>1$, also for $p=1$ we shall denote by
\begin{equation*}
    \lambda_1\leq\ldots\leq \lambda_n \qquad \text{and}\qquad \mu_1\leq\ldots\leq \mu_m
\end{equation*}the eigenvalues of $\Delta_1$ that are described in Theorem \ref{theo:minmax} and the analogous eigenvalues of $\Delta_1^H$ that can be obtained in the same way. Also in this case, as well as for $p>1$, we can always say that
\begin{equation*}
    \lambda_1=\min_{f\in C(V)}\RQ_1(f),\,\,\,\, \lambda_n=\max_{f\in C(V)}\RQ_1(f), \,\,\,\, \mu_1=\min_{\gamma\in C(H)}\RQ_1(\gamma), \,\,\,\, \mu_m=\max_{\gamma\in C(H)}\RQ_1(\gamma).
\end{equation*}
\begin{remark}
In contrast to the case of the $p$-Laplacian for $p>1$, the converse of Theorem \ref{thm:1-Laplace} is not true, that is, there exist eigenfunctions $f$ of $\Delta_1$ that are not a critical points of $\RQ_1$. However, showing this requires a long  argument that we bring forward in \cite{pLaplacians2}. In \cite{pLaplacians2} we also show, furthermore, that Conjecture \ref{conj-minmax} cannot hold for $\Delta_1$. (We had already noted in Section \ref{section:Euclidean} that this is also a subtle issue in the continuous case.) \end{remark}

\section{Smallest and largest eigenvalues}\label{section:smallestlargest}
In \cite{Sharp}, it has been proved that 
\begin{equation}\label{eq:sharp}
    \max\limits_{\gamma\in C(H)}\RQ_1(\gamma)=\max\limits_{h\in H}\sum_{i\in h}\frac{1}{\deg(i)},
\end{equation}
Hence, we can characterize the maximal eigenvalue of $\Delta^H_1$ in virtue of a combinatorial quantity. In this section we investigate further properties of both the largest and the smallest eigenvalues of the $p$-Laplacians, for general $p$.

\begin{lemma}\label{Lemma:1bounds}
        For all $p$, $\lambda_1\leq 1\leq \lambda_n$.
\end{lemma}
\begin{proof}
Let $\tilde{f}:V\rightarrow\mathbb{R}$ that is $1$ on a fixed vertex and $0$ on all other vertices. Then, for all $p$, $\RQ_p(\tilde{f})=1$. Therefore,
\begin{equation*}
       \lambda_1 =\min\limits_{f\in C(V)}\RQ_p(f)\leq \RQ_p(\tilde{f})=1 \leq \max\limits_{f\in C(V)}\RQ_p(f)=\lambda_n.
        \end{equation*}
\end{proof}
\begin{lemma}\label{lemma:rq1}
        For $p=1$ and for all hypergraphs, $\lambda_n=1.$

\end{lemma}
\begin{proof}We generalize the proof of \cite[Lemma 8]{Cheeger-like-graphs}. Let $\hat{f}:V\rightarrow\mathbb{R}$ be a maximizer of
           \begin{equation*}
\frac{\sum_{h\in H}\left|\sum_{i\text{ input of }h}f(i)-\sum_{j\text{ output of }h}f(j)\right|}{\sum_{i\in V}\deg(i)|f(i)|}        \end{equation*}and assume, without loss of generality, that $\sum_{i\in V}\deg(i)|\hat{f}(i)|=1$. Then,
        \begin{align*}
           \lambda_n&= \max_{f:V\rightarrow\mathbb{R}}\frac{\sum_{h\in H}\left|\sum_{i\text{ input of }h}f(i)-\sum_{j\text{ output of }h}f(j)\right|}{\sum_{i\in V}\deg(i)|f(i)|} \\&=\sum_{h\in H}\left|\sum_{i\text{ input of }h}\hat{f}(i)-\sum_{j\text{ output of }h}\hat{f}(j)\right|\\
            &\leq \sum_{h\in H}\sum_{i\in h}|\hat{f}(i)|\\
            &=\sum_{i\in V}\deg(i)\cdot \bigl|\hat{f}(i)\bigr|\\
            &=1.
        \end{align*}The inverse inequality follows by Lemma \ref{Lemma:1bounds}.
\end{proof}
\begin{remark}If we compare \eqref{eq:sharp} and Lemma \ref{lemma:rq1} we can see that, while for $p=2$, i.e.\ in the case of the usual hypergraph Laplacian, $\mu_m=\lambda_n$ and $\mu_1=\lambda_1$, this is not necessarily true for all $p$.
\end{remark}
\begin{lemma}\label{lemma:rqgamma1}
        For all $p$,  \begin{equation*}
       \mu_1\leq \min\limits_{h\in H}\sum_{i\in h}\frac{1}{\deg(i)}\leq \max\limits_{h\in H}\sum_{i\in h}\frac{1}{\deg(i)}\leq 
        \mu_m.
        \end{equation*}
\end{lemma}
\begin{proof}
let $\tilde{\gamma}:H\rightarrow\mathbb{R}$ that is $1$ on a fixed hyperedge $h$ and $0$ on all other hyperedges. Then, for all $p$, $$\RQ_p(\tilde{\gamma})=\sum_{i\in h}\frac{1}{\deg(i)}.$$ Therefore,
\begin{equation*}
       \mu_1= \min\limits_{\gamma\in C(H)}\RQ_p(\gamma)\leq \RQ_p(\tilde{\gamma}) =\sum_{i\in h}\frac{1}{\deg(i)}\leq 
        \max\limits_{\gamma\in C(H)}\RQ_p(\gamma)=\mu_m.
        \end{equation*}
Since this is true for all $h$, this proves the claim.
\end{proof}
\section{Nodal domain theorems}\label{section:nodal}
In \cite{MulasZhang}, the authors prove two nodal domain theorems for $\Delta_2$. In this section we establish similar results for $\Delta_p$, for all $p\geq 1$. Before, we recall the definitions of nodal domains for oriented hypergraphs. We refer the reader to \cite{nodalgraphs} for nodal domain theorems on graphs. 

\begin{definition}[\cite{MulasZhang}]\label{def:nodal-domain}
Given a function $f:V\to \R$, we let $\mathrm{supp}(f):=\{i\in V: f(i)\ne0\}$ be the support set of $f$. A \textbf{nodal domain} of $f$ is a connected component of
\begin{equation*}
    H\cap \mathrm{supp}(f):=\{h'=(h_{in}\cap \mathrm{supp}(f), h_{out}\cap \mathrm{supp}(f)): h\in H\}.
\end{equation*}
Similarly, we let $\mathrm{supp}_\pm(f):=\{i\in V: \pm f(i)>0\}$. A \textbf{positive nodal domain} of $f$ is a connected component of
\begin{equation*}
    H\cap \mathrm{supp}_+(f):=\{h'=(h_{in}\cap \mathrm{supp}_+(f), h_{out}\cap \mathrm{supp}_+(f)): h\in H\}.
\end{equation*}A \textbf{negative nodal domain} of $f$ is a connected component of $H\cap \mathrm{supp}_-(f)$.
\end{definition}
\subsection{Signless nodal domain}
\begin{definition}
We say an eigenvalue $\lambda$ of $\Delta_p$ has \textbf{multiplicity $r$} if $\mathrm{gen}\{\text{eigenfunctions w.r.t. }\lambda\}=r$.
\end{definition}

 \begin{theorem}\label{thm:Courant-nodal}
If $f$ is an eigenfunction of the $k$-th min-max  eigenvalue $\lambda_k(\Delta_p)$ and this has multiplicity $r$, then the number of nodal domains of $f$ is smaller than or equal to $k+r-1$.
\end{theorem}
\begin{proof}
Suppose the contrary, that is, $f$ is an eigenfunction of $\lambda_k$ with multiplicity $r$, and $f$ has at least $k+r$ nodal domains which are denoted by $V_1,\ldots,V_{k+r}$. For simplicity, we assume that
\begin{equation*}
  \lambda_1\le \ldots\le  \lambda_k=\lambda_{k+1}=\ldots=\lambda_{k+r-1}<\lambda_{k+r} \le\ldots\le \lambda_n.
\end{equation*}
Consider a linear function-space $X$ spanned by $f|_{V_1},\ldots,f|_{V_{k+r}}$, where the restriction $f|_{V_i}$ is defined by
\begin{equation*}
    f|_{V_i}(j)=\begin{cases}f(j),&\text{ if }j\in V_i,\\ 0,&\text{ if } j\not\in V_i.\end{cases}
\end{equation*}
Since $V_1,\ldots,V_{k+r}$ are pairwise disjoint, $\dim X=k+r$. Given $g\in X\setminus 0$, there exists $(t_1,\ldots,t_{k+r})\ne\vec0$ such that
\begin{equation*}
    g=\sum_{i=1}^{k+r} t_i f|_{V_i}.
\end{equation*}
 It is clear that $\|g\|_p^p=\sum_{i=1}^{k+r} |t_i|^p\|f|_{V_i}\|_p^p$. By the definition of nodal domain, each hyperedge $h$ intersects with at most one $V_i\in \{V_1,\ldots,V_{k+r}\}$, which implies that $E_p(g)=\sum_{i=1}^{k+r} |t_i|^p E_p(f|_{V_i})$. Finally, we note that for $p>1$, 
\begin{align*}
& \sum_{h\in H:\, h\cap V_{l}\ne\emptyset}\left|\sum_{i\in h_{in}}f|_{V_l}(i)-\sum_{j\in h_{out}}f|_{V_l}(j)\right|^p
= \sum_{h\in H:\, h\cap V_{l}\ne\emptyset}\left|\sum_{i\in h_{in}}f(i)-\sum_{j\in h_{out}}f(j)\right|^p
\\&= \sum_{i\in V_l}f(i)\sum_{h\ni i}\left|\sum_{h_{in}}f(i)-\sum_{h_{out}}f(j)\right|^{p-2}\left(\sum_{j\in h,o_h(i,j)=-1}f(j)-\sum_{j'\in h,o_h(i,j')=1}f(j')\right)
\\&= \sum_{i\in V_l}f(i) \lambda_k\deg(i)|f(i)|^{p-2}f(i)
= \lambda_k\sum_{i\in V_l}\deg(i)|f|_{V_l}(i)|^p=\lambda_k\|f|_{V_l}\|_p^p,
\end{align*}
which implies that $E_p(f|_{V_l})=\lambda_k\|f|_{V_l}\|_p^p$.  
For the case of $p=1$, we have 
\begin{align*}
 \sum_{h\in H:\, h\cap V_{l}\ne\emptyset}\left|\sum_{i\in h_{in}}f|_{V_l}(i)-\sum_{j\in h_{out}}f|_{V_l}(j)\right|
&= \sum_{h\in H:\, h\cap V_{l}\ne\emptyset}\left|\sum_{i\in h_{in}}f(i)-\sum_{j\in h_{out}}f(j)\right|
\\&= \sum_{h\in H:\, h\cap V_{l}\ne\emptyset}z_h\left(\sum_{i\in h_{in}}f(i)-\sum_{j\in h_{out}}f(j)\right)
\\&= \sum_{i\in V_l}f(i)\left(\sum_{h_{in}\ni i}z_h-\sum_{h_{out}\ni i}z_h\right)
\\&= \sum_{i\in V_l}f(i) \lambda_k\deg(i)z_i
\\&= \lambda_k\sum_{i\in V_l}\deg(i)|f|_{V_l}(i)|
\\&=\lambda_k\|f|_{V_l}\|_1,
\end{align*}
in which the parameters $z_h\in\mathrm{Sgn}(\sum_{i\in h_{in}}f(i)-\sum_{j\in h_{out}}f(j))$ and $z_i\in \mathrm{Sgn}(f(i))$ (cf. Remark \ref{remark:1-Lap}).

Therefore,
$$\RQ_p(g)=\frac{\sum_{i=1}^{k+r} |t_i|^p E_p(f|_{V_i})}{\sum_{i=1}^{k+r} |t_i|^p\|f|_{V_i}\|_p^p}=\lambda_k.$$
By the min-max principle for $\Delta_p$,
\begin{align*}
\lambda_{k+r}&=\min\limits_{ X'\in\mathrm{Gen}_{k+r}}\max\limits_{g'\in X'\setminus0}
\RQ_p(g')
\\&\le \max\limits_{g\in X\setminus0}
\RQ_p(g)
\\&=\lambda_k,
\end{align*}
which leads to a contradiction.
\end{proof}

\subsection{Positive and negative nodal domain theorem}\label{positive-nodal}

In this section, we show a new Courant nodal domain theorem for oriented hypergraphs with only inputs. Note that Theorem \ref{thm:Courant-nodal} doesn't hold if we replace ``nodal domains'' by ``positive and negative nodal domains''. In fact, for the connected hypergraph  $\Gamma_k:=(V,E_k)$  with $V:=\{1,\ldots,n\}$ and $$E_k:=\{\{i,j\}:i\le k\text{ and }j\ge k+1, \text{ or vice versa}\} $$ in which  we suppose that there are only inputs, the number of positive and negative nodal domains of the first eigenfunction w.r.t. $\lambda_1=0$ is $n$.
\begin{theorem}\label{thm:Courant-nodal-reverse}
Let $\Gamma=(V,H)$ be an oriented hypergraph with only inputs. If $f$ is an eigenfunction of the $k$-th min-max eigenvalue $\lambda_k$ and this has multiplicity $r$, then the number of nodal domains of $f$ is smaller than or equal to $n-k+r$.
\end{theorem}
\begin{proof}
Suppose the contrary, that is, $f$ is an eigenfunction of $\lambda_k$ with multiplicity $r$, and $f$ has at least $n-k+r+1$ nodal domains which are denoted by $V_1,\ldots,V_{n-k+r+1}$. 
Consider a linear function-space $X$ spanned by $f|_{V_1},\ldots,f|_{V_{n-k+r+1}}$, where the restriction $f|_{V_i}$ is defined by
\begin{equation*}
    f|_{V_i}(j)=\begin{cases}f(j),&\text{ if }j\in V_i,\\ 0,&\text{ if } j\not\in V_i.\end{cases}
\end{equation*}
Since $V_1,\ldots,V_{n-k+r+1}$ are pairwise disjoint, $\dim X=n-k+r+1$. For $g\in X\setminus 0$, there exists $(t_1,\ldots,t_{n-k+r+1})\ne\vec0$ such that $g=\sum_{i=1}^{n-k+r+1} t_i f|_{V_i}$. 
  By definition of positive and negative nodal domains, each hyperedge $h$ intersects at most one positive nodal domain and at most one negative nodal domain. Thus, for $l\ne l'$ and $h\in H$, $\left(\sum_{i\in h_{in}}f|_{V_l}(i)\right)\cdot\left(\sum_{i\in h_{in}}f|_{V_{l'}}(i)\right)\le 0$.\newline 
  
  Now, with a little abuse of notation we let $h=h_{in}$. For $p>1$, we have that
\begin{align*}
\sum_{h\in H}\left|\sum_{i\in h}g(i)\right|^p&= \sum_{h\in H}\left|\sum_{i\in h}\sum_{l=1}^{n-k+r+1}t_lf|_{V_l}(i)\right|^p
\\&=\sum_{h\in H}\left|\sum_{l=1}^{n-k+r+1}t_l\left(\sum_{i\in h}f|_{V_l}(i)\right)\right|^p
\\&\ge 
\sum_{h\in H}\sum_{l=1}^{n-k+r+1}|t_l|^p\left(\sum_{i\in h}f|_{V_l}(i)\right)\left|\sum_{i\in h}f(i)\right|^{p-2}\sum_{i\in h}f(i)
\\&=\sum_{l=1}^{n-k+r+1}|t_l|^p\sum_{i\in V_l}f(i)\left(\sum_{h\in H: h\ni i}\left|\sum_{j'\in h}f(j')\right|^{p-2}\left(\sum_{j'\in h}f(j')\right)\right)
\\&=\sum_{l=1}^{n-k+r+1}|t_l|^p\sum_{i\in V_l}f(i) \lambda_k\deg(i)|f(i)|^{p-2}f(i)
\\&=\lambda_k\sum_{l=1}^{n-k+r+1}|t_l|^p\sum_{i\in V_l}\deg(i)|f(i)|^p
\\&=\lambda_k\sum_{i\in V}\deg(i)|g(i)|^p,
\end{align*}
where the inequality is deduced by taking $A=\sum_{i\in h}f|_{V_l}(i)$ and $B=\sum_{i\in h}f|_{V_l'}(i)$ in the following lemma. Similarly, for $p=1$ we have
\begin{align*}
\sum_{h\in H}\left|\sum_{i\in h}g(i)\right|&= \sum_{h\in H}\left|\sum_{l=1}^{n-k+r+1}t_l\left(\sum_{i\in h}f|_{V_l}(i)\right)\right|
\\&\ge
\sum_{h\in H}\sum_{l=1}^{n-k+r+1}|t_l|\left(\sum_{i\in h}f|_{V_l}(i)\right)z_h
\\&=\sum_{l=1}^{n-k+r+1}|t_l|\sum_{i\in V_l}f(i)\left(\sum_{h\in H: h\ni i}z_h\right)
\\&=\sum_{l=1}^{n-k+r+1}|t_l|\sum_{i\in V_l}f(i) \lambda_k\deg(i)z_i
\\&=\lambda_k\sum_{l=1}^{n-k+r+1}|t_l|\sum_{i\in V_l}\deg(i)|f(i)|=\lambda_k\sum_{i\in V}\deg(i)|g(i)|,
\end{align*}
where $z_h\in\mathrm{Sgn}(\sum_{i\in h}f(i))$ and $z_i\in \mathrm{Sgn}(f(i))$.

\begin{lemma}\label{lemma:toprove}Let $p\ge 1$, and let $t,s,A,B\in \R$ with $AB\le 0$. Then,
\begin{equation}\label{eq:basic-inequality-p}
    |tA+sB|^p\ge (|t|^pA+|s|^pB)|A+B|^{p-2}(A+B).
\end{equation}
In the particular case of $p=1$, we further have $|tA+sB|\ge (|t|A+|s|B)z,\;\;\forall z\in\mathrm{Sgn}(A+B).$ 
\end{lemma}

By Lemma \ref{lemma:toprove}, it follows that $\RQ(g)\ge\lambda_k$. \newline

By the intersection property of $\mathbb{Z}_2$-genus, $X'\cap X\setminus\{0\}\ne\emptyset$ for any $X'\in\mathrm{Gen}_{k-r}$. Therefore,
\begin{align*}
\lambda_{k-r}&=\inf\limits_{ X'\in\mathrm{Gen}_{k-r}}\sup\limits_{g'\in X'\setminus0}\RQ(g') \ge \inf\limits_{ X'\in\mathrm{Gen}_{k-r}}\sup\limits_{g'\in X'\cap X\setminus0}\RQ(g')
\\&\ge \inf\limits_{ X'\in\mathrm{Gen}_{k-r}}\inf\limits_{g'\in X'\cap X\setminus0}\RQ(g') \ge \inf\limits_{ X'\in\mathrm{Gen}_{k-r}}\inf\limits_{g'\in X\setminus0}\RQ(g')
\\ &=\inf\limits_{g\in X\setminus0}\RQ(g) \ge\lambda_k.
\end{align*}
Together with $\lambda_{k-r}\le\ldots\le \lambda_{k-1}\le\lambda_k$, this implies that $\lambda_{k-r}=\ldots= \lambda_{k-1}=\lambda_k$, meaning that the multiplicity of 
$\lambda_k$ is at least $r+1$, which leads to a contradiction.
\end{proof}
It is only left to prove Lemma \ref{lemma:toprove}.
\begin{proof}[Proof of Lemma \ref{lemma:toprove}]
Without loss of generality, we may assume that $A>0>B$ and $A>B':=|B|$. In order to prove \eqref{eq:basic-inequality-p}, it suffices to show that 
$$|tA-sB'|^p\ge (|t|^pA-|s|^pB')(A-B')^{p-1},
$$
that is,
$$\left|t\frac{A}{A-B'}-s\frac{B'}{A-B'}\right|^p\ge |t|^p\frac{A}{A-B'}-|s|^p\frac{B'}{A-B'}.
$$
By the convexity of the function $t\mapsto |t|^p$, we have 
$$
\frac{A-B'}{A}\left|t\frac{A}{A-B'}-s\frac{B'}{A-B'}\right|^p+\frac{B'}{A}|s|^p\ge |t|^p,
$$
which proves \eqref{eq:basic-inequality-p}. 
Now, in order to prove the stronger inequality for $p=1$, since $z=|A+B|(A+B)^{-1}$ if $A+B\ne 0$, it suffices to focus on the case of $A+B=0$. In this case, by $|t-s|\ge \max\{|t|-|s|,|s|-|t|\}$, we have $|t-s|\ge (|t|-|s|)z$ for any $z\in[-1,1]$. Therefore, $|tA+sB|=A|t-s|\ge A(|t|-|s|)z=(|t|A+|s|B)z$. The proof is completed.
\end{proof}
\section{Smallest nonzero eigenvalue}\label{section:lambdamin}
In this section we discuss the smallest nonzero eigenvalue $\lambda_{\min}$ of $\Delta_p$, for $p\geq 1$, as a continuation of Sections 5 and 6 in \cite{MulasZhang}, which are focused on the easier study of $\lambda_{\min}$ for the $2$-Laplacian. As in \cite{MulasZhang}, we let $\mathcal{I}^h:V\to\R$ and $\mathcal{I}_i:H\to\R$ be defined by 
\begin{equation*} 
\mathcal{I}_i(h):=	\mathcal{I}^h(i):=\begin{cases} 1 & \text{ if }i\in h_{in}\\ -1 & \text{ if }i\in h_{out}\\ 0 & \text{otherwise.}
	\end{cases}
	\end{equation*}

\begin{theorem}\label{thm:smallest-nonzero}For $p\geq 1$,
\begin{equation}\label{eq:lambda-min}
    \lambda_{\min}=\min\limits_{f\in\mathrm{span}( \mathcal{I}^h:h\in H)}\frac{\sum_{h\in H}|\langle\mathcal{I}^h,f\rangle|^p}{\min\limits_{g\in \mathrm{span}( \mathcal{I}^h:h\in H)^\bot}\sum_{i\in V}\deg(i)|f(i)-g(i)|^p}=\lambda_{d+1},
\end{equation}
\begin{equation*}
    \mu_{\min}=\min\limits_{\gamma\in\mathrm{span}( \mathcal{I}_i:i\in V)}\frac{\sum_{i\in V}\frac{1}{\deg(i)}|\langle\mathcal{I}_i,\gamma\rangle|^p}{\min\limits_{\eta\in \mathrm{span}( \mathcal{I}_i:i\in V)^\bot}\sum_{h\in H} |\gamma(h)-\eta(h)|^p}=\mu_{d'+1},
\end{equation*}
where $d:=\dim \mathrm{span}( \mathcal{I}^h:h\in H)^\bot$ and $d':=\dim \mathrm{span}( \mathcal{I}_i:i\in V)^\bot$.
\end{theorem}
\begin{remark}
Equation \eqref{eq:lambda-min} above generalizes Equation (5) in \cite{MulasZhang}. In fact, for $p=2$, by letting
\begin{equation*}
    \sum_{i\in V}\deg(i)|f(i)-\bar g(i)|^2:=\min\limits_{g\in \mathrm{span}( \mathcal{I}^h:h\in H)^\bot}\sum_{i\in V}\deg(i)|f(i)-g(i)|^2
\end{equation*}we have that $\bar f:=f-\bar g$ is orthogonal to $\mathrm{span}( \mathcal{I}^h:h\in H)^\bot$ with respect to the weighted scalar product $(f',g'):=\sum_{i\in V}\deg(i) f'(i) g'(i)$. Therefore,
\begin{align*}
\lambda_{\min}&=\min\limits_{f\in\mathrm{span}( \mathcal{I}^h:h\in H)} \frac{\sum_{h\in H}\langle\mathcal{I}^h,f \rangle^2}{(\bar f,\bar f)}
\\&= 
\min\limits_{\bar f\in \mathrm{span}\{D^{-1}\mathcal{I}^h:h\in H\}}\frac{\sum_{h\in H}\langle\mathcal{I}^h,\bar f \rangle^2}{(\bar f,\bar f)}
\\&=
\min\limits_{f\in\mathrm{span}\{D^{-\frac12}\mathcal{I}^h:h\in H\}}\frac{\sum_{h\in H}\langle D^{-\frac12}\mathcal{I}^h,f \rangle^2}{\langle f,f\rangle}
\end{align*}and this coincides with  Equation (5) in \cite[Lemma 6.1]{MulasZhang}.
\end{remark}

\begin{proof}[Proof of Theorem \ref{thm:smallest-nonzero}]
Let $X:=\mathrm{span}(\mathcal{I}^h:h\in H)^\bot$. We shall prove that
\begin{equation*}
    \lambda_{\min}=\lambda_{d+1}=\tilde{\lambda}:=\min\limits_{f\in X^\bot}\frac{\sum_{h\in H}|\langle\mathcal{I}^h,f\rangle|^p}{\min\limits_{g\in X}\sum_{i\in V}\deg(i)|f(i)-g(i)|^p}.
\end{equation*}
If $d=0$, the claim is straightforward because in this case $X=0$, $X^\bot=\mathbb{R}^n$ and  $$\lambda_{\min}=\min\limits_{f\in  \mathbb{R}^n}\frac{\sum_{h\in H}|\langle\mathcal{I}^h,f\rangle|^p}{ \sum_{i\in V}\deg(i)|f(i)|^p}=\lambda_1.$$

Now, assume $d\ge 1$. Since $X\in \mathrm{Gen}_d$ and $\RQ_p(f)=0$ for all $f\in X$, we have $\lambda_1=\ldots=\lambda_d=0$. From the local compactness of $X^\bot$, the zero-homogeneity of $\RQ_p(f)$ and the fact that $E_p(f)>0$ $\forall X^\bot\setminus X$, it follows that $\tilde{\lambda}>0$. For the case when $p>1$, we still need to prove the following three steps.
\begin{enumerate}[(I)]
    \item $\lambda_{d+1}\ge \tilde{\lambda}$:\newline
    Observe that $\dim X^\bot=n-d$. Since the $l^p$-norm is smooth and strictly convex for $p>1$, for each $f$ there is a  unique $g_f\in X$ such that $$ \sum_{i\in V}\deg(i)|f(i)-g_f(i)|^p=\min\limits_{g\in X}\sum_{i\in V}\deg(i)|f(i)-g(i)|^p$$ 
    and the map $\varphi:f\mapsto f-g_f$ is smooth. Moreover, $\varphi|_{X^\bot}:X^\bot\to \varphi(X^\bot)$ is  bicontinuous (i.e., homeomorphism). Clearly, $\varphi$ is such that $-f\mapsto -f-g_{-f}=-f+g_f$, therefore $\varphi$ is odd. Hence, if we let $f^\bot$ be the projection of $f$ to $X^\bot$, we get an odd homeomorphism $\psi:R^n\to R^n$, $f\mapsto f-g_{f^\bot}$.\newline
    Thus, because of the homotopy property of the $\mathbb{Z}_2$-genus, for any $S\in \mathrm{Gen}_{d+1}$ we have that the image $\psi^{-1}(S)\in \mathrm{Gen}_{d+1}$. Moreover, by the intersection property of the $\mathbb{Z}_2$-genus, $\psi^{-1}(S)\cap X^\bot\ne\emptyset$, which implies $S\cap \psi(X^\bot)=\psi(\psi^{-1}(S)\cap X^\bot)\ne\emptyset$. Also note that $\psi(X^\bot)=\varphi(X^\bot)$.   Hence for any $S\in \mathrm{Gen}_{d+1}$, $$\sup\limits_{f\in S}\RQ_p(f)\ge \inf\limits_{f\in \varphi(X^\bot)}\RQ_p(f)=\tilde{\lambda}.$$ This proves that $\lambda_{d+1}\ge \tilde{\lambda}$. 
     \item $\lambda_{d+1}\le \tilde{\lambda}$:
     
     For any $f\in X^\bot\setminus X$, let $X':=\mathrm{span}(X\cup\{f\})$. Then, $X'\in \mathrm{Gen}_{d+1}$ and  
     $$\lambda_{d+1}\le \sup\limits_{f'\in X'}\RQ_p(f')=\sup\limits_{g\in X}\frac{E_p(f)}{\sum_{i\in V}\deg(i)|f(i)+g(i)|^p}=\frac{\sum_{h\in H}|\langle\mathcal{I}^h,f\rangle|^p}{\min\limits_{g\in X}\sum_{i\in V}\deg(i)|f(i)-g(i)|^p}.$$ Since this holds for all $f\in X^\bot$, we derive that $\lambda_{d+1}\le \tilde{\lambda}$. 
     \item There is no positive eigenvalue between $\lambda_1=0$ and $\lambda_{d+1}>0$:\newline
     Suppose the contrary and assume that $f$ is an eigenfunction with eigenvalue $\RQ_p(f)\in(0,\tilde{\lambda})$. Then $\nabla \RQ_p(f)=0$.  Consider the function $t\mapsto \RQ_p(f-tg_f)$. On the one hand, $$\frac{d}{dt}\biggl|_{t=0}\RQ_p(f-tg_f)=-\langle\nabla \RQ_p(f),g_f\rangle=0.$$
  On the other hand, $E_p(f-tg_f)=E_p(f)$ and the function 
  \begin{equation}\label{function:t}
      t\mapsto \sum_{i\in V}\deg(i)|f(i)-tg_f(i)|^p
  \end{equation}
   is a \emph{strictly convex} function with minimum at $t=1$. This implies that \eqref{function:t} is strictly  decreasing and convex on $(-1,1)$, thus $$\frac{d}{dt}\biggl|_{t=0}\,\sum_{i\in V}\deg(i)|f(i)-tg_f(i)|^p<0.$$ Hence, we get  $\frac{d}{dt}\bigl|_{t=0}\RQ_p(f-tg_f)>0$, which leads to a contradiction. 
\end{enumerate}
This proves the case $p>1$. Finally, we complete the proof of the case $p=1$. Since
\begin{equation*}
    \lambda_{d+1}(\Delta_p)\xrightarrow{p\rightarrow 1^+} \lambda_{d+1}(\Delta_1)\qquad \text{and}\qquad \tilde{\lambda}(\Delta_p)\xrightarrow{p\rightarrow 1^+} \tilde{\lambda}(\Delta_1),
\end{equation*}we only need to prove that (III) holds also for $\Delta_1$. Suppose the contrary and let $\hat{f}$ be an eigenfunction corresponding to an eigenvalue $\lambda\in(0,\tilde{\lambda})$. Then, $0\in\nabla E_1(\hat{f})-\lambda \nabla \|\hat{f}\|_1$. Now, consider a flow near $\hat{f}$ defined by $\eta(f,t):=f-tg_f$, where $t\ge 0$ and $f\in \mathbb{B}_\delta(f)$ for sufficiently small $\delta>0$. Note that $$ E_1(f-tg_f)-\lambda  \|f-tg_f\|_1= E_1(f)-\lambda  \|f-tg_f\|_1$$ is an increasing function of $t$, since $\|f-tg_f\|_1<\|f\|_1$ and $\|\cdot\|_1$ is convex. Consequently, by the theory of weak slope \cite{Liusternik}, we have that $0\not\in \nabla (E_1(\hat{f})-\lambda \|\hat{f}\|_1)=\nabla E_1(\hat{f})-\lambda \nabla \|\hat{f}\|_1$, which is a contradiction. This completes the proof.
\end{proof}
We shall now discuss some consequences of Theorem \ref{thm:smallest-nonzero}.

\begin{corollary}\label{cor:lambda-min-}For $p\geq 1$,
    \begin{equation*}
    \lambda_{\min}\ge \min\limits_{f\in\mathrm{span}( \mathcal{I}^h:h\in H)}\frac{\sum_{h\in H}|\langle\mathcal{I}^h,f\rangle|^p}{ \sum_{i\in V}\deg(i)|f(i)|^p}.
\end{equation*}
\end{corollary}
\begin{proof}
It follows immediately from Theorem \ref{thm:smallest-nonzero}.
\end{proof}
 \begin{corollary} For $p\geq 1$, let $\lambda_{p,\min}$ be the smallest positive eigenvalue of the $p$-Laplacian. Then,
    \begin{equation*}
    \lambda_{p,\min}\ge \begin{cases}
    |H|^{1-\frac p2} \lambda_{2,\min}^{\frac p2}  &\text{ if }p\ge2,
    \\ \vol(V)^{\frac p2-1}\lambda_{2,\min}^{\frac p2}  &\text{ if }p\le2.
    \end{cases}
\end{equation*}
\end{corollary}
\begin{proof} For $p\le2$, it is known that $\sum_{h\in H}|\langle\mathcal{I}^h,f\rangle|^p\ge \left(\sum_{h\in H}|\langle\mathcal{I}^h,f\rangle|^2\right)^{p/2}$ and 
$$\left(\frac{\sum_{i\in V}\deg(i)|f(i)|^p}{\vol(V)}\right)^{\frac1p}\le \left(\frac{\sum_{i\in V}\deg(i)|f(i)|^2}{\vol(V)}\right)^{\frac12}.$$
Thus, applying Corollary \ref{cor:lambda-min-}, we have $$ \lambda_{p,\min}\ge\min\limits_{f\in\mathrm{span}( \mathcal{I}^h:h\in H)}\vol(V)^{\frac p2-1}\left(\frac{\sum_{h\in H}|\langle\mathcal{I}^h,f\rangle|^2}{ \sum_{i\in V}\deg(i)|f(i)|^2}\right)^{\frac p2}=\vol(V)^{\frac p2-1}\lambda_{2,\min}^{\frac p2}.$$
The case of $p\ge 2$ is similar. 
\end{proof}

\begin{remark}\label{remark:lambda-min-pq}
We further have 
    \begin{equation*}
    \frac{\hat{\lambda}_{p,\min}}{\hat{\lambda}_{q,\min}}\ge \begin{cases}
    |H|^{\frac 1p-\frac 1q}  &\text{ if }p\ge q,
    \\ \vol(V)^{\frac 1q-\frac 1p}   &\text{ if }p\le q,
    \end{cases}
\end{equation*}
where $\hat{\lambda}_{p,\min}=\lambda_{p,\min}^{\frac1p}$. This implies that 
\begin{equation*}
    \vol(V)^{\frac 1q-\frac 1p}\ge \frac{\hat{\lambda}_{p,\min}}{\hat{\lambda}_{q,\min}}\ge  |H|^{\frac 1p-\frac 1q} \qquad\text{ if }p\ge q,
\end{equation*}thus $\hat{\lambda}_{p,\min}$ is a continuous function of $p\in [1,\infty)$ and the limit $\lim\limits_{p\to+\infty}\hat{\lambda}_{p,\min}\in[0,n]$ exists.
\end{remark}
\begin{remark}
For $p\geq 1$, let $$C_p:=\max\limits_{f\in\mathrm{span}( \mathcal{I}^h:h\in H)}\frac{ \sum_{i\in V}\deg(i)|f(i)|^p}{\min\limits_{g\in \mathrm{span}( \mathcal{I}^h:h\in H)^\bot}\sum_{i\in V}\deg(i)|f(i)-g(i)|^p }.$$
By Corollary \ref{cor:lambda-min-} and Remark \ref{remark:lambda-min-pq}, we get that
   \begin{equation*}
    \lambda_{\min}\le C_p \cdot \min\limits_{f\in\mathrm{span}( \mathcal{I}^h:h\in H)}\frac{\sum_{h\in H}|\langle\mathcal{I}^h,f\rangle|^p}{ \sum_{i\in V}\deg(i)|f(i)|^p},
\end{equation*}which can be seen as a dual inequality with respect to the one in Corollary \ref{cor:lambda-min-}. Note that the constant $C_p$ is such that $C_2=1$ for all oriented hypergraphs and $C_1=2$ in the graph case.
\end{remark}

\section{Vertex partition problems}\label{section:vertexpartition}
In \cite{MulasZhang}, two vertex partition problems for oriented hypergraphs have been discussed: the $k$-coloring, that is, a function $f:V\rightarrow\{1,\ldots,k\}$ such that $f(i)\neq f(j)$ for all $i\neq j\in h$ and for all $h\in H$, and the generalized Cheeger problem. In this section we discuss more partition problems and we also define a new coloring number that takes signs into account as well.\newline

In \cite{MulasZhang}, the generalized Cheeger constant is defined as
	\begin{equation*}
		h:=\min_{\emptyset\neq S:\vol S\leq \frac{1}{2}\vol \overline{S}}\frac{e(S)}{\vol(S)},
		\end{equation*}where, given $\emptyset\neq S\subseteq V$,
			\begin{equation*}
		e(S):=\sum_{h\in H}\biggl(\#(S\cap h_{in})-\#(S\cap h_{out})\biggr)^2,\end{equation*}$\overline{S}:=V\setminus S$ and
		\begin{equation*}
		    \vol(S):=\sum_{i\in S}\deg (i).
		\end{equation*}We generalize $e(S)$ by letting, for $p\geq 1$ and $\emptyset\neq S\subseteq V$,
		\begin{equation*}
		    e_p(S):=\sum_{h\in H}|\#(S\cap h_{in})-\#(S\cap h_{out})|^p.
		\end{equation*}
		\begin{remark}For a graph, $e_p(S)=e(S)=|\partial(S)|$ is the number of edges between $S$ and $\overline{S}$, for all $p$. It measures, therefore, the \emph{flow} between $S$ and $\overline{S}$. More generally, we can say that computing $e_p(S)$ (as well as $\vol S$) means \emph{deleting} all vertices in $\overline{S}$, in the sense of \cite[Definition 2.20]{MulasZhang}, and then computing $e_p$ (respectively, the volume) on the vertex set of the sub-hypergraph obtained. Furthermore, when $e_p$ is computed on the vertex set, 
		\begin{equation*}
   0 \leq e_p(V)=\sum_{h\in H}|\#h_{in}-\#h_{out}|^p\leq \sum_{h\in H}|\#h|^p,
\end{equation*}where the first inequality is an equality if and only if $\#h_{in}=\#h_{out}$ for each hyperedge, and the second one is an equality if and only if there are either only inputs or only outputs. Hence, we could see the case $e_p(V)=0$ as a \emph{balance condition}. Having $\#h_{in}=\#h_{out}$ means that \emph{what comes in is the same as what goes out}. Hence, also in the general case we can say that $e_p$ measures a flow.
		\end{remark}
		\subsection{$k$-cut problems}
We now generalize the \emph{balanced minimum $k$-cut problem} and the \emph{max $k$-cut problem}, known for graphs \cite{kcut,maxcut}, to the case of hypergraphs.
		\begin{definition}
		Given $k\in\{2,\ldots,n\}$, the \textbf{balanced minimum $k$-cut} is
		\begin{equation*}
		    \min\limits_{\text{partition }(V_1,\ldots,V_k)}\sum_{i=1}^k\frac{e_p(V_i)}{\vol(V_i)}.
		\end{equation*}The \textbf{maximum $k$-cut} is
		\begin{equation*}
		    \max\limits_{\text{partition }(V_1,\ldots,V_k)}\sum_{i=1}^k e_p(V_i).
		\end{equation*}
		\end{definition}
\begin{lemma}\label{lemma:ep}
        For each $\emptyset\neq S\subseteq V$ and for each $p\geq 1$,
        \begin{equation*}
           \lambda_1\leq \frac{e_p(S)}{\vol S}\leq \lambda_n.
        \end{equation*}Therefore, in particular, for each $k\in\{2,\ldots,n\}$
        \begin{equation*}
            \lambda_n\geq \frac1k\cdot \max\limits_{\text{partition }(V_1,\ldots,V_k)}\sum_{i=1}^k\frac{e_p(V_i)}{\vol(V_i)}\geq \frac{1}{k\cdot \vol(V)}\max\limits_{\text{partition }(V_1,\ldots,V_k)}\sum_{i=1}^k e_p(V_i)
        \end{equation*}and
        \begin{equation*}
             \lambda_1\leq \frac1k\cdot \min\limits_{\text{partition }(V_1,\ldots,V_k)}\sum_{i=1}^k\frac{e_p(V_i)}{\vol(V_i)}.
        \end{equation*}
\end{lemma}
\begin{proof}
Let $f\in C(V)$ be $1$ on $S$ and $0$ on $\bar{S}$. Then, 
\begin{equation*}
    \RQ_p(f)=\frac{e_p(S)}{\vol (S)}.
\end{equation*}The second claim follows by applying the first one to all the $V_i$'s.
\end{proof}
\subsection{Signed coloring number}
We now introduce the new notion of \emph{signed coloring number}, that takes into account also the input/output structure of the hypergraph. We denote by $\chi(\Gamma)$ the coloring number defined in \cite{MulasZhang}.
\begin{definition}A \textbf{signed $k$-coloring of the vertices} is a function $f:V\to \{1,\ldots,k\}$ such that, for all $h\in H$, $f(i)\ne f(j)$ if $i$ and $j$ are anti-oriented in $h$. The \textbf{signed coloring number} of $\Gamma$, denoted $\chi_{\sgn}(\Gamma)$ is the minimal $k$ such that there exists a signed $k$-coloring.
\end{definition}
\begin{remark}
    Note that $\chi_{\sgn}(\Gamma)\leq \chi(\Gamma)$. Also, $\chi_{\sgn}\leq2$ if and only if $\Gamma$ is bipartite. 
\end{remark} Applying Lemma \ref{lemma:ep} to the signed coloring number, we get the following corollary.
\begin{corollary}
Let $\chi_{\sgn}:=\chi_{\sgn}(\Gamma)$ and let $V_1,\ldots,V_{\chi_{\sgn}}$ be the corresponding coloring classes. For each $p\geq 1$,
\begin{equation}\label{eq:signedcolor}
           \lambda_1\leq \frac{1}{\chi_{\sgn}}\Biggl(\sum_{i=1}^{\chi_{\sgn}}\sum_{h\in H}\frac{|\#(V_i\cap h)|^p}{\vol(V_i)}\Biggr) \leq \lambda_n.
        \end{equation}
\end{corollary}Also, the upper bound in \eqref{eq:signedcolor} shrinks to an equality for $p=1$.
\begin{proof}The first fact follows from Lemma \ref{lemma:ep} since, by definition of signed coloring number, 
\begin{equation*}
    |\#(V_i\cap h_{in})-\#(V_i\cap h_{out})|=\#(V_i\cap h),
\end{equation*}for each coloring class $V_i$.\newline

In the particular case of $p=1$, $\sum_{h\in H}\#(V_i\cap h)=\vol(V_i)$ for each $i$, therefore
\begin{equation*}
   \frac{1}{\chi_{\sgn}}\Biggl(\sum_{i=1}^{\chi_{\sgn}}\sum_{h\in H}\frac{\#(V_i\cap h)}{\vol(V_i)}\Biggr)=1.
\end{equation*}Since we know, from Lemma \ref{lemma:rq1}, that $\max_f\RQ_1(f)=1$, this proves that the upper bound in \eqref{eq:signedcolor} shrinks to an equality for $p=1$.
\end{proof}
\begin{remark}The fact that the upper bound in \eqref{eq:signedcolor} shrinks to an equality for $p=1$ is particularly interesting because this is similar to what happens for the Cheeger constant $h$ in the case of graphs, and for the Cheeger-like constant $Q$ defined in \cite{Cheeger-like-graphs} for graphs and generalized in \cite{Sharp} for hypergraphs. In fact, we have that:
\begin{enumerate}
    \item For connected graphs, the Cheeger constant $h$ can be used for bounding $\lambda_2$ in the case of $\Delta_2$ and, as shown in \cite{Hein2,Chang}, it is equal to $\lambda_2$ in the case of $\Delta_1$.
    \item For general hypergraphs, the Cheeger-like constant $Q$ can be used for bounding $\lambda_n$ in the case of $\Delta_2$ and $\Delta^H_2$, and it is equal to $\lambda_n$ in the case of $\Delta^H_1$ (cf. \cite{Sharp}).
    \item In \eqref{eq:signedcolor} we again have something similar, because the quantity that bounds $\lambda_n$ from below for $\Delta_p$ equals $\lambda_n$ for $\Delta_1$.
\end{enumerate}Of course, the main difference between the last case and the first two is that $h$ and $Q$ are constants that are independent of $p$, while the quantity in \eqref{eq:signedcolor} changes when $p$ changes.
\end{remark}
\begin{remark}
In the case of graphs, by definition of signed coloring number we have that $\#(V_i\cap h)\in\{0,1\}$ for each coloring class $V_i$ and for each each edge $h$. In particular, $$\sum_{e\in E}|\#(V_i\cap e)|^p= \vol(V_i)$$ and the constant appearing in \eqref{eq:signedcolor} is equal to $1$ for all $p$.
\end{remark}
\subsection{Multiway partitioning}
In this section we generalize the notion of $k$-cut and we use it for bounding the smallest and largest eigenvalue of the classical Laplacian $\Delta_2$.

\begin{definition}
A $k$-tuple $(S_1,\ldots,S_k)$ of sets $S_r\subseteq V$ is called a \textbf{$(k,l)$-family} if it covers $S_1\cup\ldots \cup S_k$ exactly $l$-times (i.e., each vertex $i\in S_1\cup\ldots \cup S_k$ lies in exactly $l$ sets $S_{i_1},\ldots,S_{i_l}$). If, furthermore, $S_1\cup\ldots\cup S_k=V$, then we call the $(k,l)$-family a \textbf{$(k,l)$-cover}. 
\end{definition}
\begin{remark}
A $(k,1)$-cover is a $k$-partition (or $k$-cut).
\end{remark}
\begin{theorem}\label{theo:kl}
Let $\lambda_1$ and $\lambda_n$ be the smallest and the largest eigenvalue of the classical normalized Laplacian $\Delta_2$, respectively. For any $(k,l)$-family,
\begin{equation*}
    \lambda_1\leq \frac{k\cdot (e(S_1)+\ldots+ e(S_k))-l^2\cdot e(S_1\cup\ldots \cup S_k)}{(k-l)\cdot l\cdot \vol(S_1\cup\ldots \cup S_k)} \leq \lambda_n.
\end{equation*}
\end{theorem}

\begin{proof}
We first focus on the case that $(S_1,\ldots,S_k)$ is a $(k,l)$-cover. 
 For $r\in \{1,\ldots,k\}$, define a function $f_r:V\to\R$ by
$$f_r(i)=\begin{cases}t,&\text{ if }i\in S_r,\\
s,&\text{ if }i\not\in S_r.\end{cases}$$
Then 
\begin{align*}
&\left(\sum_{j\in h_{in}} f_r(j)-\sum_{j'\in h_{out}} f_r(j')\right)^2
\\=~&
(t\#(S_r\cap h_{in})-t\#(S_r\cap h_{out})+s(\#h_{in}\setminus S_r)-s(\#h_{out}\setminus S_r))^2
\\=~&
(t\#(S_r\cap h_{in})-t\#(S_r\cap h_{out})+s\#h_{in}-s\#(S_r\cap h_{in})-s\#h_{out}+s\#(S_r\cap h_{out}))^2
\\=~& ((t-s)(\#(S_r\cap h_{in})-\#(S_r\cap h_{out}))+s(\#h_{in}-\#h_{out}))^2
\\=~&(t-s)^2(\#(S_r\cap h_{in})-\#(S_r\cap h_{out}))^2+s^2(\#h_{in}-\#h_{out})^2
\\&+2(t-s)s(\#(S_r\cap h_{in})-\#(S_r\cap h_{out})(\#h_{in}-\#h_{out}).
\end{align*}

Consequently,
\begin{align*}
&\sum_{r=1}^k\sum_{h\in H} \left(\sum_{j\in h_{in}} f_r(j)-\sum_{j'\in h_{out}} f_r(j')\right)^2
\\=~&(t-s)^2\sum_{r=1}^k\sum_{h\in H}(\#(S_r\cap h_{in})-\#(S_r\cap h_{out}))^2+s^2k\sum_{h\in H}(\#h_{in}-\#h_{out})^2
\\&+2(t-s)s\sum_{h\in H}\sum_{r=1}^k(\#(S_r\cap h_{in})-\#(S_r\cap h_{out})(\#h_{in}-\#h_{out})
\\=~&(t-s)^2\sum_{r=1}^k\sum_{h\in H}(\#(S_r\cap h_{in})-\#(S_r\cap h_{out}))^2+(s^2k+2(t-s)sl)\sum_{h\in H}(\#h_{in}-\#h_{out})^2
\\=~&(t-s)^2\sum_{r=1}^k e(S_r)+(s^2k+2(t-s)sl)e(V),
\end{align*}
where we have used the equality 
$$\sum_{r=1}^k(\#(S_r\cap h_{in})-\#(S_r\cap h_{out})=l(\#h_{in}-\#h_{out}),$$ since each vertex in $h$ is covered  $l$ times by $S_1,\ldots,S_k$ ($l\le k$). \newline

Also, $\sum_{i\in V} \deg(i)f_r(i)^2=\vol(S_r)t^2+(\vol(V)-\vol(S_r))s^2$ and $\sum_{r=1}^k\vol(S_r)=l\vol(V)$. Hence,
$$\sum_{r=1}^k\sum_{i\in V} \deg(i)f_r(i)^2=l\vol(V)(t^2-s^2)+k\vol(V)s^2.$$
By the basic inequality
$$
\lambda_n\ge \frac{\sum_{r=1}^k\sum_{h\in H} \left(\sum_{j\in h_{in}} f_r(j)-\sum_{j'\in h_{out}} f_r(j')\right)^2}{\sum_{r=1}^k\sum_{i\in V} \deg(i)f_r(i)^2} \ge\lambda_1,
$$
we have 
$$
\eta(t,s):=\frac{(t-s)^2\sum_{r=1}^k e(S_r)+(s^2k+2(t-s)sl)e(V)}{\vol(V)(l(t^2-s^2)+k s^2)}
\in[\lambda_1,\lambda_n].$$
We can verify that the minimum and maximum of the above quantity belong to 
$$\left\{\frac{\sum_{r=1}^k e(S_r)}{\vol(V)}\frac{k}{l(k-l)}-\frac{e(V)}{\vol(V)}\frac{l}{k-l},\;\; 
\frac{e(V)}{\vol(V)}\right\}.$$
To see this, we make the following observations.
\begin{enumerate}
\item For $s=0$, the  quantity $\eta(t,s)$ is $\frac{\sum_{r=1}^ke_r(S_r)}{l\vol(V)}$.
\item For $s\ne 0$, the  quantity $\eta(t,s)$ is
$$
\frac{(\frac ts-1)^2\frac1l\sum_{r=1}^k e(S_r)+(\frac kl+2(\frac ts-1))e(V)}{\vol(V)((\frac ts)^2-1+\frac kl)}=\frac{e(V)}{\vol(V)}+\frac{(\frac ts-1)^2}{(\frac ts)^2-1+\frac kl}\frac{\frac1l\sum_{r=1}^k e(S_r)-e(V)}{\vol(V)}.
$$
In fact, since $\max\limits_{(t,s)\ne (0,0)}\frac{(\frac ts-1)^2}{(\frac ts)^2-1+\frac kl}=\frac{k}{k-l}$ and $\min\limits_{(t,s)\ne (0,0)}\frac{(\frac ts-1)^2}{(\frac ts)^2-1+\frac kl}=0$, $k\ge l$, 
$$\{\max\limits_{s\ne0,t\in\mathbb{R}}\eta(t,s),\min\limits_{s\ne0,t\in\mathbb{R}}\eta(t,s)\} = \left\{\frac{e(V)}{\vol(V)},\; \frac{e(V)}{\vol(V)}+ \frac{k}{k-l}\frac{\frac1l\sum_{r=1}^k e(S_r)-e(V)}{\vol(V)}\right\}.$$
\end{enumerate}
The proof of the claim is then completed by observing that
$$\frac{\sum_{r=1}^ke_r(S_r)}{l\vol(V)}=\frac lk \frac{e(V)}{\vol(V)}+(1-\frac lk)\left(\frac{e(V)}{\vol(V)}+ \frac{k}{k-l}\frac{\frac1l\sum_{r=1}^k e(S_r)-e(V)}{\vol(V)}\right).$$

For a general $(k,l)$-family, we can consider $V':=S_1\cup\ldots\cup S_k$ and $H':=\{h\cap V':h\in H\}$. Then, $\Gamma':=(V',H')$ is a sub-hypergraph of $H$ restricted to $V'$. According to \cite[Lemma 2.21]{MulasZhang}, $\lambda_n(\Gamma)\ge \lambda_{\max}(\Gamma')$ and $\lambda_1(\Gamma)\le \lambda_1(\Gamma')$. Applying the case of the $(k,l)$-cover to $\Gamma'$, we complete the proof. 
\end{proof}
\begin{corollary}
$$
\lambda_1\leq \frac{k(e(S_1)+\ldots+ e(S_k))-l^2e(V)}{(k-l)l\vol(V)}\leq \lambda_n.
$$
\end{corollary}
\begin{remark}For a graph, $e(S)=|\partial(S)|$ and a $(2,1)$-cover is a standard $2$-cut. Theorem \ref{theo:kl} shows that
$$\lambda_n\ge 4\max\limits_{S\subset V}\frac{|\partial(S)|}{\vol(V)},$$  where $2\max\limits_{S\subset V}\frac{|\partial(S)|}{\vol(V)}$ is the normalized max-cut ratio.\newline

Also, Theorem \ref{theo:kl} applied to $(2,1)$-families for a graph implies that
\begin{align*}
    \lambda_n&\ge \max\limits_{S_1\cap S_2=\emptyset}\frac{2(e(S_1)+e(S_2))-e(S_1\cup S_2)}{\vol(S_1)+\vol(S_2)}\\
    &=\max\limits_{S_1\cap S_2=\emptyset}\frac{4|E(S_1,S_2)|+|\partial(S_1\cup S_2)|}{\vol(S_1)+\vol(S_2)}
\\&\ge 2\max\limits_{S_1\cap S_2=\emptyset}\frac{2|E(S_1,S_2)|}{\vol(S_1)+\vol(S_2)},
\end{align*}
where $\max\limits_{S_1\cap S_2=\emptyset}\frac{2|E(S_1,S_2)|}{\vol(S_1)+\vol(S_2)}$ is exactly the dual Cheeger constant \cite{dual}. \newline

Interestingly, applying Theorem \ref{theo:kl} to  a $(k,1)$-cover of a graph, we get 
$$\lambda_n\ge \frac{k}{k-1}\cdot \frac{\max\limits_{\text{partition }(V_1,\ldots,V_k)}\sum_{i=1}^k|\partial V_i|}{\vol(V)}$$
which relates to the max $k$-cut problem. 
\end{remark}
\subsection{General partitions}\label{section:generalpartitions}

\begin{lemma}\label{lemma:p-lambda-1-n}
We have
\begin{equation}\label{eq:lambda-n-lower}
\lambda_n(\Delta_p)\ge \max\limits_{t\in\R, c\in[0,1]}\max\limits_{\text{partition }(V_1,\ldots,V_k)}\frac{c^{p-1}|t+1|^p\sum_{r=1}^k e_p(V_r)-(\frac{c}{1-c})^{p-1}ke_p(V)}{\vol(V)(|t|^p+k-1)} 
\end{equation}
and
\begin{equation}\label{eq:lambda-1-upper}
\lambda_1(\Delta_p)\le \min\limits_{t\in\R}\min\limits_{\text{partition }(V_1,\ldots,V_k)}\frac{|t+1|^p\sum_{r=1}^k e_p(V_r)+ke_p(V)}{\vol(V)(|t|^p+k-1)}.
\end{equation}
\end{lemma}

\begin{proof}Given a partition $(V_1,\ldots,V_k)$ of $V$,  given $r\in \{1,\ldots,k\}$, define a function $f_r:V\to\R$ by
$$f_r(i):=\begin{cases}t&\text{ if }i\in V_r,\\
-1&\text{ if }i\not\in V_r.\end{cases}$$
Then,
\begin{align}
&\left|\sum_{j\in h_{in}} f_r(j)-\sum_{j'\in h_{out}} f_r(j')\right|^p \notag
\\=~&
\left|t\#(V_r\cap h_{in})-t\#(V_r\cap h_{out})-\#h_{in}\setminus V_r+\#h_{out}\setminus V_r\right|^p \notag
\\=~&|(t+1)(\#(V_r\cap h_{in})-\#(V_r\cap h_{out}))-(\#h_{in}-\#h_{out})|^p \label{eq:(t+1)}
\\ \le~& |t+1|^p|\#(V_r\cap h_{in})-\#(V_r\cap h_{out})|^p+|\#h_{in}-\#h_{out}|^p. \notag
\end{align}
Consequently,
\begin{align*}
&\sum_{r=1}^k\sum_{h\in H} \left|\sum_{j\in h_{in}} f_r(j)-\sum_{j'\in h_{out}} f_r(j')\right|^p
\\ \le~&\sum_{r=1}^k\sum_{h\in H}|t+1|^p|\#(V_r\cap h_{in})-\#(V_r\cap h_{out})|^p+|\#h_{in}-\#h_{out}|^p
\\=~&|t+1|^p\sum_{r=1}^k e_p(V_r)+ke_p(V).
\end{align*}
Also, we have
$$\sum_{r=1}^k\sum_{i\in V} \deg(i)|f_r(i)|^p=\vol(V)(|t|^p-1)+k\vol(V).$$
Now, note that
$$\lambda_1\le \frac{\sum_{r=1}^k\sum_{h\in H} \left|\sum_{j\in h_{in}} f_r(j)-\sum_{j'\in h_{out}} f_r(j')\right|^p}{\sum_{r=1}^k\sum_{i\in V} \deg(i)|f_r(i)|^p}\le \frac{|t+1|^p\sum_{r=1}^k e_p(V_r)+ke_p(V)}{\vol(V)(|t|^p+k-1)}.$$

Next, we give a lower bound for \eqref{eq:(t+1)}. By the convexity of $t\mapsto |t|^p$, we have
$$c\left|\frac{1}{c}(B-A)\right|^p+(1-c)\left|\frac{1}{1-c}A\right|^p \ge |B|^p \qquad \forall A,B\in \R,\,0<c<1,$$ which implies  $|B-A|^p\ge c^{p-1}|B|^p-(\frac{c}{1-c})^{p-1}|A|^p$. Thus, 
\begin{align*}
&|(t+1)(\#(V_r\cap h_{in})-\#(V_r\cap h_{out}))-(\#h_{in}-\#h_{out})|^p
\\ \ge~& c^{p-1}|(t+1)(\#(V_r\cap h_{in})-\#(V_r\cap h_{out}))|^p-(\frac{c}{1-c})^{p-1}|(\#h_{in}-\#h_{out})|^p.
\end{align*}
Finally, the same method gives
$$\frac{c^{p-1}|t+1|^p\sum_{r=1}^k e_p(V_r)-(\frac{c}{1-c})^{p-1}ke_p(V)}{\vol(V)(|t|^p+k-1)}\le\lambda_n.$$
\end{proof}
\begin{corollary}
The following constants are smaller than or equal to $\lambda_n(\Delta_p)$:
$$\frac{2(\frac k2)^{p-1}\sum_{r=1}^k e_p(V_r)-ke_p(V)}{((k-1)^{p}+k-1)\vol(V)},\;\frac{k\sum_{r=1}^k e_p(V_r)-\frac{k}{(k-1)^{p-1}}e_p(V)}{((k-1)^{p}+k-1)\vol(V)},\;\frac{2\sum_{r=1}^k e_p(V_r)-ke_p(V)}{k\vol(V)}.$$
\end{corollary}

\begin{proof}
Taking $t=k-1$ and $c=\frac12$ in \eqref{eq:lambda-n-lower}, we have the first. 

Taking $t=k-1$ and $c=\frac1k$ in \eqref{eq:lambda-n-lower}, we get the middle one. 

Taking $t=1$ and $c=\frac12$ in \eqref{eq:lambda-n-lower}, we obtain the last one. 
\end{proof}

\begin{corollary}
The following constants are larger than or equal to $\lambda_1(\Delta_p)$:
$$ \frac{e_p(V)}{\vol(V)},\; \frac{\sum_{r=1}^ke_p(V_r)}{\vol(V)}\;.$$
\end{corollary}

\begin{proof}
Taking $t=-1$ in \eqref{eq:lambda-1-upper}, we get the first constant. Letting $t\to\infty$ in \eqref{eq:lambda-1-upper}, we obtain the second one.
\end{proof}

\begin{corollary}\label{cor:1-Lap-estimate}
$$
\lambda_n(\Delta_1)\ge \max\limits_{\text{partition }(V_1,\ldots,V_k)}\frac{ \sum_{r=1}^k e_1(V_r) }{\vol(V) }\;\;\text{ and }\;\; \lambda_1(\Delta_1)\le \min\limits_{\text{partition }(V_1,\ldots,V_k)}\frac{ \sum_{r=1}^k e_1(V_r) }{\vol(V) }
$$
\end{corollary}
\begin{proof}
Taking $p=1$ in Lemma \ref{lemma:p-lambda-1-n}, we have 
\begin{align*}
   \lambda_n(\Delta_1)&\ge \max\limits_{t\in\R, c\in[0,1]}\max\limits_{\text{partition }(V_1,\ldots,V_k)}\frac{|t+1|\sum_{r=1}^k e(V_r)-ke(V)}{\vol(V)(|t|+k-1)}  
   \\&= \max\limits_{\text{partition }(V_1,\ldots,V_k)}\max_{t\in\R}\frac{|t+1|\sum_{r=1}^k e(V_r)-ke(V)}{\vol(V)(|t|+k-1)}  
   \\&=\max\limits_{\text{partition }(V_1,\ldots,V_k)}\frac{ \sum_{r=1}^k e(V_r) }{\vol(V) }
\end{align*}
and
\begin{align*}
    \lambda_1(\Delta_1)&\le \min\limits_{t\in\R}\min\limits_{\text{partition }(V_1,\ldots,V_k)}\frac{|t+1|\sum_{r=1}^k e_1(V_r)+ke_1(V)}{\vol(V)(|t|+k-1)}
    \\&=\min\left\{\frac{e_1(V)}{\vol(V)},\min\limits_{\text{partition }(V_1,\ldots,V_k)}\frac{ \sum_{r=1}^k e(V_r) }{\vol(V) } \right\}
    .
\end{align*}

\end{proof}
\section{Hyperedge partition problems}\label{section:hyperedge}
While in the previous section we have discussed vertex partition problems and their relation to $\Delta_p$, here we introduce the analogous hyperedge partition problems and their relations with $\Delta_p^H$. We start by defining, for each $\emptyset\neq \hat{H}\subset H$, a quantity $e_p(\hat{H})$ analogous to the quantity $e_p(S)$ defined for subsets of vertices. Namely, we let
\begin{equation*}
    e_p(\hat{H}):=\sum_{i\in V}\frac{1}{\deg (i)}|\#(\hat{H}\cap i_{in})-\#(\hat{H}\cap i_{out})|^p
\end{equation*}where, given $i\in V$, we let
\begin{equation*}
    i_{in}:=\#\{\text{ hyperedges in which $i$ is an input }\},\qquad i_{out}:=\#\{\text{ hyperedges in which $i$ is an output }\}.
    \end{equation*}We also define
    \begin{equation*}
        \eta_p(\hat{H}):=\frac{e_p(\hat{H})}{\#\hat{H}}.
    \end{equation*}
\begin{remark}Analogously to the vertex case, we can say that computing $e_p(\hat{H})$ means \emph{deleting} all hyperedges in $H\setminus \hat{H}$ and then computing $e_p$ on the hyperedge set of the sub-hypergraph obtained. It is therefore interesting to observe that, when $e_p$ is computed on $H$,
\begin{equation*}
    0\leq e_p(H)=\sum_{i\in V}\frac{1}{\deg (i)}\cdot |\# i_{in}-\# i_{out}|^p\leq \sum_{i\in V}\deg (i)^{p-1},
\end{equation*}where the first inequality is an equality if and only if each vertex is as often an input as an output, while the second one is an equality if and only if all vertices have the same sign for all hyperedges in which the graph is contained.\newline

Furthermore, if the sub-hypergraph $\hat{\Gamma}:=(V,\hat{H})$ of $\Gamma$ is bipartite, without loss of generality we can assume that each vertex is either always an input or always an output for each hyperedge in which it is contained. In this case,
\begin{equation*}
    e_p(\hat{H})=\sum_{i\in V}\frac{\deg_{\hat{\Gamma}}(i)^p}{\deg (i)}
\end{equation*}and, in particular, $\eta_p(\hat{H})$ coincides with the quantity in \cite[Definition 2.9]{Sharp}. Moreover, in the particular case when $\hat{H}=\{h\}$ is given by one single hyperedge, then
\begin{equation*}
    e_p(\{h\})=\eta_p(\{h\})\sum_{i\in h}\frac{1}{\deg i} \text{ for all }p.
\end{equation*}
\end{remark}
We now generalize \cite[Lemma 4.1]{Sharp} for all $p$. 

\begin{proposition}For all $p$, we have that
	 \begin{equation*}
	     \max_{\hat{\Gamma}=(V,\hat{H})\subset\Gamma \text{ bipartite}} \eta_p(\hat{H})\leq \mu_m,
	 \end{equation*}with equality if $p=1$.
	 \end{proposition}
	 \begin{proof}Let $\gamma':H\rightarrow\mathbb{R}$ be $1$ on $\hat{H}$ and $0$ otherwise. Then, up to changing (without loss of generality) the orientations of the hyperedges,
		\begin{align*}
		\mu_m&=\max_{\gamma:H\rightarrow\mathbb{R}}\frac{\sum_{i\in V}\frac{1}{\deg (i)}\cdot \biggl(\sum_{h_{\text{in}}: i\text{ input}}\gamma(h_{\text{in}})-\sum_{h_{\text{out}}: i\text{ output}}\gamma(h_{\text{out}})\biggr)^p}{\sum_{h\in H}|\gamma(h)|^p}\\
					&\geq \frac{\sum_{i\in V}\frac{1}{\deg (i)}\cdot \biggl(\sum_{h_{\text{in}}: v\text{ input}}\gamma'(h_{\text{in}})-\sum_{h_{\text{out}}: v\text{ output}}\gamma'(h_{\text{out}})\biggr)^p}{\sum_{h\in H}|\gamma'(h)|^p}\\
					&\geq \frac{\sum_{i\in \hat{V}}\frac{1}{\deg (i)}\cdot \biggl(\sum_{h_{\text{in}}: i\text{ input}}\gamma'(h_{\text{in}})-\sum_{h_{\text{out}}: i\text{ output}}\gamma'(h_{\text{out}})\biggr)^p}{\sum_{h\in H}|\gamma'(h)|^p}\\
					&= \frac{\sum_{i\in \hat{V}}\frac{\deg_{\hat{\Gamma}}(i)^p}{\deg (i)}}{|\hat{H}|}=\eta_p(\hat{H}).				\end{align*}Since the above inequality is true for all $\hat{\Gamma}$, this proves the first claim.\newline
					
					If $p=1$, then 
					\begin{equation*}
					   \mu_m\geq \max_{\hat{\Gamma}=(V,\hat{H})\text{ bipartite}}\eta_p(\hat{H})\geq \max_{h\in H}\eta_p(\{h\})=Q,
					\end{equation*}where $Q$ is the Cheeger-like quantity defined in \cite{Sharp}. By \cite[Lemma 5.2]{Sharp}, $Q=\mu_m$. Therefore the last inequalities shrink to equalities.
	 \end{proof}Now, analogously to the vertex partition problems, we discuss hyperedge partition problems.
\begin{definition}
A \textbf{$k$-hyperedge partition} is a partition of the hyperedge set into $k$ disjoint sets, $H=H_1\sqcup\ldots\sqcup H_k$. The \textbf{balanced minimum $k$-hyperedge cut} is
\begin{equation*}
    \min_{\text{partition }(H_1,\ldots,H_k)}\sum_{i=1}^k\eta_p(H_i);
\end{equation*}The \textbf{maximum $k$-hyperedge cut} is
\begin{equation*}
    \max_{\text{partition }(H_1,\ldots,H_k)}\sum_{i=1}^ke_p(H_i).
\end{equation*}The \textbf{signed hyperedge coloring number}, denoted $\chi_{\sgn}^H$, is the minimal $k$ for which there exists a function $\gamma:H\rightarrow \{1,\ldots,k\}$ such that, for all $i\in V$, $\gamma(h)\neq \gamma(h')$ if $i$ is an input for $h$ and an output for $h'$.
\end{definition}The following lemma is the analog of some results regarding vertex partition problems. It relates the balanced minimum $k$-hyperedge cut and the maximum $k$-hyperedge cut to the smallest and largest eigenvalues of $\Delta_p^H$, respectively.
\begin{lemma}\label{lemma:edgepartitions}
      For each $\emptyset\neq \hat{H}\subseteq H$ and for each $p\geq 1$, $\mu_1\leq \eta_p(\hat{H})\leq \mu_m$. Therefore, in particular, for each $k\in\{2,\ldots,n\}$
        \begin{equation*}
            \mu_m\geq \frac1k\cdot \max\limits_{\text{partition }(H_1,\ldots,H_k)}\sum_{i=1}^k\eta_p(H_i)\geq \frac{1}{k\cdot \# H}\max\limits_{\text{partition }(H_1,\ldots,H_k)}\sum_{i=1}^k e_p(H_i)
        \end{equation*}and
        \begin{equation*}
             \mu_1\leq \frac1k\cdot \min\limits_{\text{partition }(V_1,\ldots,V_k)}\sum_{i=1}^k\eta_p(H_i).
        \end{equation*}
\end{lemma}
\begin{proof}
Given $H_i$, let $\gamma\in C(H)$ be $1$ on $H_i$ and $0$ otherwise. Then, $\RQ_p(\gamma)=\eta_p(H_i)$. Therefore, $\mu_1\leq \eta_p(H_i)\leq \mu_m$. The other claims follow by applying these inequalities to all elements of a partition.
\end{proof}

\begin{corollary}
Let $\chi_{\sgn}^H$ be the signed hyperedge coloring number of $\Gamma$ and let $H_1,\ldots,H_{\chi_{\sgn}^H}$ be the corresponding coloring classes. Let also $\Gamma_j:=(V,H_j)$ for $j\in\{1,\ldots,\chi_{\sgn}^H\}$. For each $p\geq 1$,
\begin{equation*}
           \mu_1\leq \frac{1}{\chi_{\sgn}^H}\Biggl(\sum_{j=1}^{\chi_{\sgn}^H}\frac{1}{\# H_j}\cdot\sum_{i\in V}\frac{\#(H_j\cap i)^p}{\deg (i)}\Biggr) \leq \mu_m.
        \end{equation*}
\end{corollary}
\begin{proof}By definition of signed hyperedge coloring number, 
\begin{equation*}
     e_p(H_j)=\sum_{i\in V}\frac{1}{\deg (i)}|\#(H_j\cap i_{in})-\#(H_j\cap i_{out})|^p=\sum_{i\in V}\frac{\#(H_j\cap i)^p}{\deg (i)},
\end{equation*}for each coloring class $H_j$. Together with Lemma \ref{lemma:edgepartitions}, this proves the claim.
\end{proof}

\section*{Acknowledgments}
We thank Friedemann Schuricht and two anonymous referees for valuable suggestions and references. 

\bibliography{pLaplacians}	

\end{document}